\documentclass[12pt]{article}

\usepackage{amsthm}
\usepackage{amsmath}
\usepackage{amsfonts}
\usepackage{amssymb}
\usepackage{graphicx}
\usepackage{color}
\usepackage[english]{babel}
\usepackage{soul}
\usepackage[utf8]{inputenc}
\usepackage{tkz-graph}
\usetikzlibrary{decorations.markings}

\usepackage{tikz}
\usetikzlibrary{graphs}
\usepackage{relsize}

\theoremstyle{plain}   
\newtheorem{theorem}{Theorem}[section]
\newtheorem{proposition}[theorem]{Proposition}
\newtheorem{lemma}[theorem]{Lemma}

\DeclareMathOperator{\Aut}{Aut}
\def\Cay{\mbox{\rm Cay}}

\DeclareMathOperator{\id}{id}
\def\Z{\ns Z}

\def\u{\mbox{\boldmath $u$}}

\def\x{\mbox{\boldmath $x$}}
\def\y{\mbox{\boldmath $y$}}
\def\z{\mbox{\boldmath $z$}}

\def\vecu{\mbox{\boldmath $u$}}
\def\vecv{\mbox{\boldmath $v$}}

\def\vec0{\mbox{\boldmath $0$}}

\def\A{\mbox{\boldmath $A$}}

\def\I{\mbox{\boldmath $I$}}

\def\I{\mbox{\boldmath $I$}}

\def\Z{\mbox{\boldmath $Z$}}
\def\1{\mbox{\boldmath $1$}}

\newcommand{\dist}{\mathrm{dist}}

\title{On large regular $(1,1,k)$-mixed graphs}

\author{
	C. Dalf\'o$^a$, G. Erskine$^b$, G. Exoo$^c$, M. A. Fiol$^d$,\\
 N. L\'opez$^a$, A. Messegu\'e$^a$, J. Tuite$^b$\\
	\\
	{\small $^a$Dept. de Matem\`atica, Universitat de Lleida, Catalonia}\\
	{\small {\tt \{cristina.dalfo,nacho.lopez\}@udl.cat}, {\tt  arnau.messegue@upc.edu}}\\
	{\small $^{b}$School of Mathematics and Statistics, Open University, Milton Keynes, UK}\\
	{\small {\tt \{grahame.erskine,james.t.tuite\}@open.ac.uk}}\\
{\small $^{c}$Dept. of Mathematics and Computer Science, Indiana State University, USA} \\
           {\small {\tt  ge@cs.indstate.edu}}\\
	{\small $^{d}$Dept. de Matem\`atiques, Universitat Polit\`ecnica de Catalunya, Barcelona, Catalonia} \\
	{\small Barcelona Graduate School of Mathematics} \\
           {\small Institut de Matem\`atiques de la UPC-BarcelonaTech (IMTech)}\\
           {\small {\tt miguel.angel.fiol@upc.edu}}}

\date{}
\begin{document}
	
\maketitle

\begin{abstract}
An  $(r,z,k)$-mixed graph $G$ has every vertex with undirected degree $r$, directed in- and out-degree $z$, and diameter $k$. In this paper, we  study the case $r=z=1$, proposing some new constructions of $(1,1,k)$-mixed graphs with a large number of vertices $N$. 
Our study is based on computer techniques for small values of $k$
and the use of graphs on alphabets for general $k$. In the former case, the constructions are either Cayley or lift graphs. In the latter case, some infinite families of $(1,1,k)$-mixed graphs are proposed with diameter of the order of  $2\log_2 N$.
\end{abstract}



\section{Introduction}
The relationship between vertices or nodes in interconnection networks can be undirected or directed depending on whether the communication between nodes is two-way or only one-way. Mixed graphs arise in this case and in many other practical situations, where both kinds of connections are needed. Urban street networks are perhaps the most popular ones. 
Thus, a {\em mixed graph } $G=(V,E,A)$ has a set $V=V(G)=\{u_1,u_2,\ldots\}$ of vertices, a set $E=E(G)$ of edges or unordered pairs of vertices $\{u,v\}$, for $u,v\in V$, and a set $A=A(G)$ of arcs, directed edges, or ordered pair of vertices $uv\equiv (u,v)$. 
For a given vertex $u$, its {\em undirected degree} $r(u)$  is the number of edges incident to vertex $u$.
Moreover, its {\em out-degree} $z^+(u)$ is the number of arcs emanating from $u$, whereas its {\em in-degree} $z^-(u)$  is the number of arcs going to $u$.
If $z^+(u)=z^-(u)=z$ and $r(u)=r$, for all $u \in V$, then $G$ is said to be a {\em totally regular} $(r,z)$-mixed graph with {\em whole degree} $d=r+z$.

The distance from vertex $u$ to vertex $v$ is denoted by $\dist(u,v)$. Notice that, when the out-degree $z$ is not zero, the distance $\dist(u,v)$ is not necessarily equal to the distance $\dist(v,u)$.
 If the mixed graph $G$ has diameter $k$, its {\em distance matrix} $\A_i$, for $i=0,1,\ldots,k$, has entries
$(\A_i)_{uv}=1$ if $\dist(u,v)=i$, and $(\A_i)_{uv}=0$ otherwise. So, $\A_0=\I$ (the identity matrix) and $\A_1=\A$ (the adjacency matrix of $G$).

Mixed graphs were first considered in the context of the degree/diameter
problem by Bos\'ak \cite{b79}. 
The {\em degree/diameter problem} for mixed graphs reads as follows:
Given three natural numbers $r$, $z$, and $k$, find the largest possible number of vertices $N(r,z,k)$ in a
mixed graph $G$ with maximum undirected degree $r$, maximum directed out-degree $z$, and diameter $k$.

For mixed graphs, an upper bound for $N(r,z,k)$, known as a {\em Moore$($-like$)$ bound} $M(r,z,k)$, was obtained by Buset, El Amiri, Erskine, Miller, and P\'erez-Ros\'es \cite{baemp15} (also by Dalf\'o, Fiol, and L\'opez \cite{dfl18} with an alternative computation). 

\begin{theorem}[Buset, El Amiri, Erskine, Miller, and P\'erez-Ros\'es \cite{baemp15}]
\label{th:Moore}
The Moore bound for an $(r,z)$-mixed graph with diameter $k$ is
\begin{equation}
\label{eq:Moore}
M(r,z,k)=A\frac{u_1^{k+1}-1}{u_1-1}+B\frac{u_2^{k+1}-1}{u_2-1},
\end{equation}
where
\begin{align*}
u_1 &=\displaystyle{\frac{z+r-1-\sqrt{v}}{2}},\qquad u_2=\displaystyle{\frac{z+r-1+\sqrt{v}}{2}},\\
A &=\displaystyle{\frac{\sqrt{v}-(z+r+1)}{2\sqrt{v}}},\qquad
B=\displaystyle{\frac{\sqrt{v}+(z+r+1)}{2\sqrt{v}}},\\
v &=(z+r)^2+2(z-r)+1.
\end{align*}
\end{theorem}

This bound applies whether or not $G$ is totally regular, but it is elementary to show that a Moore mixed graph must be totally regular.
Thus, a Moore $(r,z,k)$-mixed graph is a graph with diameter $k$, maximum undirected degree $r\geq 1$, maximum out-degree $z\geq 1$, and order given by $M(r,z,k)$.
An example of a Moore $(3,1,2)$-mixed graph is the Bos\'ak graph \cite{b79}, see Figure \ref{Bosak}.

  \begin{figure}[t]
    \begin{center}
        \includegraphics[width=5cm]{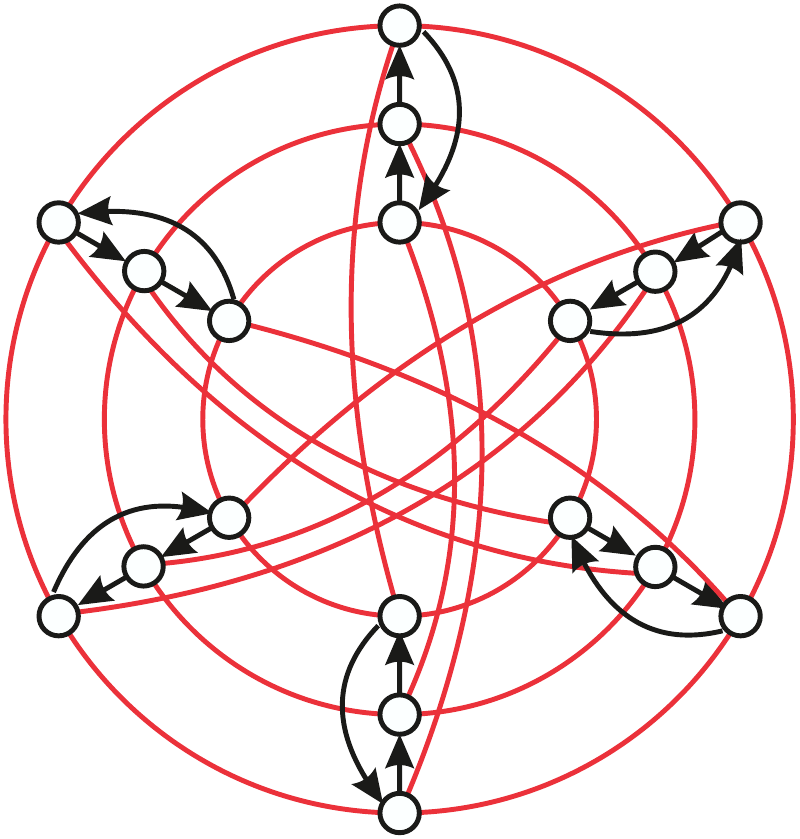}
    \end{center}
    \vskip-.5cm
    \caption{The Bos\'ak $(3,1)$-graph with diameter $k=2$ and $N=18$ vertices.}
    \label{Bosak}
\end{figure}

Bos\'ak \cite{b79} gave a necessary condition for the existence of a mixed Moore graph with diameter $k=2$. Such graphs have the property that for any ordered pair $(u,v)$ of vertices, there is a {\em unique walk of length at most $2$} between them.  In general, there are infinitely many pairs $(r,z)$ satisfying Bos\'ak necessary condition for which the existence of a mixed Moore graph is not known yet. Nguyen, Miller, and Gimbert \cite{nmg07} proved the existence and unicity of some  Moore mixed graphs of diameter 2.
L\'opez, Miret, and Fern\'andez, \cite{lmf15} proved that there is no Moore $(r,z,2)$-mixed graph when the pair $(r,z)$ equals $(3,3)$, $(3,4)$, or $(7,2)$.

For diameter $k \geq 3$, it was proved that mixed Moore graphs do not exist, see Nguyen, Miller, and Gimbert \cite{nmg07}. In the case of total regularity, this result also follows from the improved bound in Dalf\'o, Fiol, and L\'opez \cite{dfl18}, where it was shown that the order $N$ of an $(r, z)$-regular mixed graph G with diameter $k\ge 3$ satisfies 
\begin{equation}
N\le M(r,z,k)-r,
\label{improved-Moore}
\end{equation}
where $M(r,z,k)$ is given by \eqref{eq:Moore}.
In general, a mixed graph with maximum undirected degree $r$, maximum directed out-degree $z$, diameter $k$, and order $N=M(r,z,k)-\delta$ is said to have 
{\em defect} $\delta$. A mixed graph with defect one is called an {\em almost mixed Moore
graph}. Thus, the result in \eqref{improved-Moore} can be rephrased by saying that $r$ is a lower bound for the defect of the mixed graph.
In the case $r=z=1$, such a result was drastically improved by Tuite and Erskine \cite{te22} by showing that a lower bound $\delta(k)$ for the defect of a $(1,1)$-regular mixed graph with diameter $k\ge 1$ satisfies the recurrence
\begin{equation}
\delta(k+6)=\delta(k)+f_{k-1}+f_{k+4},
\label{improved-defect-recur}
\end{equation}
where the initial values of $\delta(k)$, for $k=1,\ldots,6$, are $0,1,1,2,3,5$, and $f_k$ are the Fibonacci numbers starting from $f_0=f_1=1$, namely,
$1,1,2,3,5,8,13,21,\ldots$ Alternatively, starting from $\delta(1)=0$ and $\delta(2)=1$, we have $\delta(k+2)=\delta(k+1)+\delta(k)$ if $k+2\not\equiv 1,2 \ (\textrm{mod}\ 6)$, and $\delta(k+2)=\delta(k+1)+\delta(k)+1$, otherwise.

For more results on degree/diameter problem for graphs, digraphs, and mixed graphs, see the comprehensive survey by Miller and \v{S}ir\'a\v{n} \cite{ms13}. 
For more results on mixed graphs, see Buset, L\'opez, and Miret \cite{blm17},
Dalf\'o \cite{Da19}, Dalf\'o, Fiol, and L\'opez \cite{dfl17,dfl18,dfl18bis}, Erskine \cite{e17}, J{\o}rgensen \cite{J15},
L\'opez, P\'erez-Ros\'es, and Pujol\`as \cite{lpp14}, Nguyen, Miller, and Gimbert \cite{nmg07}, and Tuite and Erskine \cite{tg19}.

In this paper, we deal with $(1,1,k)$-mixed graphs, that is,  mixed graphs with undirected degree $r=1$, directed out-degree $z=1$, and with diameter $k$.
Our study is based on computer techniques for small values of $k$,
and the use of graphs on alphabets for general $k$. In the former case, the constructions are either Cayley or lift graphs. In the latter case, some infinite families of $(1,1,k)$-mixed graphs are proposed with $N$ vertices and diameter $k$ of the order of  $2\log_2 N$. Most of the proposed constructions are closely related to line digraphs. Given a digraph $G$, its line digraph $LG$ has vertices representing the arcs of $G$, and vertex $x_1x_2$ is adjacent to vertex $y_1y_2$ in $LG$
if the arc $(x_1,x_2)$ is adjacent to the arc $(y_1,y_2)$ in $G$, that is, if $y_1=x_2$. The $k$-iterated line digraph is defined recursively as $L^kG=L^{k-1}(LG)$. Let $K_d^+$ be the complete symmetric
digraph with $d$ vertices with loops, and $K_{d+1}$ the complete symmetric digraph on $d+1$ vertices (in these complete graphs each edge is seen as a digon, or pair of opposite  arcs). Then, two well know families of iterated line digraphs are the De Bruijn digraphs $B(d,k)=L^k(K_d^+)$, and the Kautz digraphs $K(d,k)=L^k(K_{d+1})$.
Both $B(d,k)$ and $K(d,k)$ have diameter $k$ but De Bruijn digraphs have $d^k$ vertices, whereas Kautz digraphs have $d^k+d^{k-1}$ vertices.
See, for instance, 
Fiol, Yebra, and Alegre \cite{FiYeAl84}, and Miller and \v{S}ir\'a\v{n} \cite{ms13}.

\section{Some infinite families of $(1,1,k)$-mixed graphs}

In this section, we propose some infinite families of $(1,1,k)$-mixed graphs with exponential order. All of them have vertices with out-degree $z=1$. When, moreover, all the vertices have in-degree $1$, we refer to them as $(1,1,z)$-regular mixed graphs.
If we denote by $f(r,z,k)$ the order of a largest $(r,z,k)$-mixed graph, which is upper bounded by the (exponential) Moore bound $M(r,z,k)$, all the described graphs provide exponential lower bounds for $f(1,1,k)$.

Let us first give some basic properties of $(1,1,k)$-mixed graphs.
It is readily seen that the Moore bound satisfies the Fibonacci-type recurrence
$$
M(1,1,k)=M(1,1,k-1)+M(1,1,k-2)-2,
$$
starting from $M(1,1,0)=1$ and $M(1,1,1)=3$. From this, or just applying \eqref{eq:Moore}, we obtain that the corresponding Moore bound is
\begin{equation}
\label{eq:Moore(1,1,k)}
M(1,1,k)=\left(1-\frac{2}{\sqrt{5}}\right)\left(\frac{1-\sqrt{5}}{2}\right)^{k+1}+\left(1+\frac{2}{\sqrt{5}}\right)\left(\frac{1+\sqrt{5}}{2}\right)^{k+1}-2.
\end{equation}
The obtained values for $k=2,\ldots,16$ are shown in Table \ref{table4}. Then, for large values of $k$, the Moore bound $M(1,1,k)$
is of the order of
$$
M(1,1,k) \sim \left(1+\frac{2}{\sqrt{5}}\right)\left(\frac{1+\sqrt{5}}{2}\right)^{k+1}\approx 1.8944\cdot 1.6180^{k+1}.
$$

\subsection{The mixed graphs $E(n)$}

The first construction is the simplest one.
Given $n\ge 2$, the graph $E(n)$ is defined as follows.
As before, label the Fibonacci numbers so that $f_0 = f_1 = 1$.  Consider a Moore tree of radius
$n$ with base vertex $u_0$.  The set of vertices at distance $i$ from $u_0$ is 
referred to as the vertices at {\em level $i$}.  There are $f_{i+1}$ vertices at level $i$.
We can partition these vertices into two sets: $V_i$ contains the $f_{i}$ vertices at level $i$ incident through an arc from level $i-1$, and $W_i$ contains the $f_{i-1}$ vertices at level $i$ incident through an edge from level $i-1$.

To complete the graph, we must consider two cases, depending on whether $f_n$ is even or odd.
If $f_n$ is even, then we add a matching among the vertices of $V_n$ and add an arc from each
vertex in level $n$ to $u_0$.  In this case, the diameter is $2n$.  Note that the maximum distance occurs from a level-1 vertex to a level-t vertex on the opposite edge (where the two edges are based on the two level-1 vertices).

If $f_n$ is odd, then we must modify the construction slightly.
In this case, when we add a matching among the vertices of $V_n$, there is one
vertex $v_1$ of $V_n$ missed by the matching. So, we must add another vertex $v_2$, join this vertex
to $v_1$ by an edge, and then add an arc from $v_2$ to the base vertex $u_0$.  All other vertices
at level $n$ have arcs directly to $u_0$.
In this case, the diameter is $2n+1$, where the maximum distance
occurs from a level-1 vertex (in the edge {\em not} containing $v_1$) to $v_2$.

So, the graph $E(n)$ has diameter $2n$ or $2n+1$, and order $M(1,1,n)$ or $M(1,1,n) + 1$.
This bound is very weak for small diameters but at least it gives a first explicit
construction that gives an exponential lower bound.  
In the following subsections, we show that we can do it better.

\subsection{The mixed graphs $F(n)$}

Given $n\ge 2$, the $(1,1,k)$-mixed graph $F(n)$ has vertices labeled with
$a:x_1\ldots x_n$, where $a\in\{+1,-1\}$, with $x_i\in \mathbb{Z}_3$, and $x_{i+1}\neq x_i$ for $i=1,\ldots,n-1$. The adjacencies are as follows:
\begin{itemize}
	\item[$(i)$] $a:x_1x_2\ldots x_n\,\sim\, -a:x_1x_2\ldots x_n$ (edges);
	\item[$(ii)$] $a:x_1x_2\ldots x_n\,\rightarrow\, a:x_2x_3\ldots x_n(x_n+a)$ (arcs).
\end{itemize}
Thus, $F(n)$ has $3\cdot 2^n$ vertices, it is an out-regular graph but not in-regular since
vertices $a:x_1x_2\ldots x_n$ and $a:x'_1x_2\ldots x_n$ are both adjacent to $a:x_2\ldots (x_n+a)$.

\begin{figure}[t]
    \begin{center}
        \includegraphics[width=10cm]{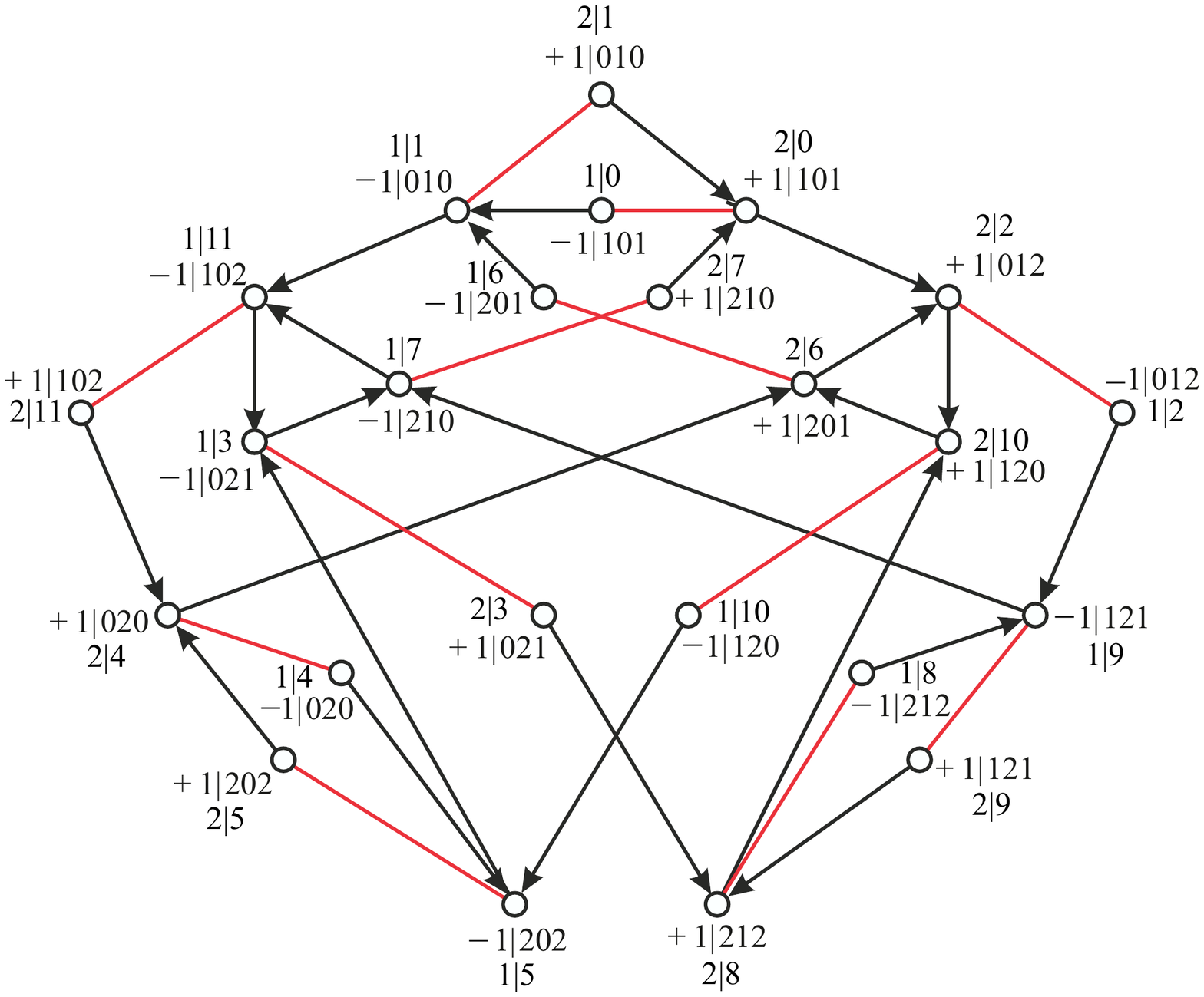}
    \end{center}
    \vskip-6.25cm
	\caption{The mixed graph $F(3)$.}
	\label{fig:F(3)}
\end{figure}

The mixed graph $F(3)$ is shown in Figure \ref{fig:F(3)}.
It is easily checked that the mapping
\begin{equation}
a|x_1x_2\ldots x_n \quad \mapsto\quad -a|\overline{x_1}\,\overline{x_2}\ldots \overline{x_n},
\label{isomorfF(n)}
\end{equation}
where $\overline{0}=1$, $\overline{1}=0$, and $\overline{2}=2$, is an automorphism of $F(n)$. This is because $\overline{x_n+a}=\overline{x_n}-a$.

\begin{proposition}
\label{diam-F(n)}
The diameter of the mixed graph $F(n)$ is $k=2n$.
\end{proposition}
\begin{proof}
Let us see that there is a path of length at most $2n$ from vertex $\x=a|x_1\ldots x_n$ to vertex $\y=b|y_1\ldots y_n$. Taking into account the automorphism in \eqref{isomorfF(n)}, we can assume that $a=+1$.
Notice that, with at most two steps, depending on the values of $a$, $x_n$, and $y_1$ (at the beginning) or $a$, $y_i$, and $y_{i+1}$ (in the sequel), we can add a new digit of $\y$. Thus, in principle, we would need at most $2n$ steps but, possibly, one last step to fix the first digit to the one of $\y$ (for example, $b$). However, in what follows, we show that the first two digits $y_1$ and $y_2$ can be `placed', so reaching a vertex of the form $a'|\ldots y_1y_2$, with at most 3 steps.
\begin{itemize}
\item
If $y_1=x_n$, the first step is not necessary.
\item 
If $y_1=x_n+1$, go through the arc $+1|x_1\ldots x_n\rightarrow +1|x_2\ldots x_ny_1$.
\item
If $y_1=x_n-1$ and $y_2=y_1-1$, go through the edge and two arcs
$$
+1|x_1\ldots x_n\sim -1|x_1x_2\ldots x_n\rightarrow -1|x_2\ldots x_ny_1
\rightarrow -1|x_3\ldots x_ny_1y_2.
$$
\item
If $y_1=x_n-1(=x_n+2)$ and $y_2=y_1+1$, go through the three arcs
\begin{align*}
+1|x_1\ldots x_n &\rightarrow +1|x_2\ldots x_nx_n+1\rightarrow 
+1|x_3\ldots x_n+1x_n+2\\
 &=+1|x_3\ldots x_n+1y_1
 \rightarrow +1|x_4\ldots x_ny_1y_2.
\end{align*}
\end{itemize}
Thus, to reach $\y$, we need at most $3+2(n-2)+1=2n$ steps.
Finally, it is not difficult to find vertices that are at distance $2n$.
For instance, for $n$ odd, go from $\x=+1|0101\ldots 0$ to
$\y=-1|2020\ldots 2$; and, for $n$ even, go from $\x=+1|1010\ldots 0$ to
$\y=+1|2020\ldots 0$,
\end{proof}

\subsubsection{A numeric construction}

An alternative presentation $F[n]$ of $F(n)$ is as follows:
Given $n\ge 1$, let $N'=3\cdot2^{n-1}$ so that the number of vertices of $F[n]$ is $2N'$.
The vertices of $F[n]$ are labeled as $\alpha|i$, where $\alpha\in\{1,2\}$, and $i\in \mathbb{Z}_{N'}$.
Let $\overline{1}=2$ and $\overline{2}=1$. Then, the adjacencies of $F[n]$ defining the same mixed graph as those in $(i)$ and $(ii)$ are:
\begin{align}
\alpha|i\,& \sim\, \overline{\alpha}|i \mbox{\  (edges);}\\
\alpha|i\,&\rightarrow\, \alpha|-2i+\alpha \mbox{\  (arcs).}
\end{align}
To show that both constructions give the same mixed graph, $F[n]\cong F(n)$,
define first the mapping $\pi$ from the two digits $x_1x_2$  to $\mathbb{Z}_6$ as follows
$$
\pi(01)=0,\ \pi(10)=1,\ \pi(12)=2,\ \pi(21)=3,\ \pi(20)=4, \pi(02)=5.
$$
Then, it is easy to check that, for $n=2$, the mapping $\psi$ from the vertices of $F(2)$ to the vertices of $F[2]$ defined as
$$
\psi(a|x_1x_2)=\alpha(a)|\pi(x_1x_2),
$$
where $\alpha(a)=\frac{a+3}{2}$, is an isomorphism from $F(2)$ to $F[2]$.
(Note that $\alpha(-1)=1$ and $\alpha(+1)=2$).
From this, we can use induction. First, let us assume that $\psi'$ is an isomorphism from $F(n-1)$ to $F[n-1]$ of the form
$$
\psi'(a|x_1x_2\ldots x_{n-1})=\alpha(a)|\pi'(x_1x_2\ldots x_{n-1}),
$$
where the linear mapping $\alpha$ is defined as above, and $\pi'$ is a mapping from the sequences $x_1x_2\ldots x_{n-1}$ to the elements of $\mathbb{Z}_{N'}$, with $N'=3\cdot 2^{n-1}$.
Then, we claim that the mapping $\psi$ from the vertices of $F(n)$ to the vertices of $F[n]$ defined as
\begin{equation}
\label{psi-F(n)}
\psi(a|x_1x_2\ldots x_n)=\alpha(a)|-2\cdot\pi'(x_1x_2\ldots x_{n-1})+\alpha(x_n-x_{n-1}) \ (\textrm{mod}\ N),
\end{equation}
where $N=3\cdot 2^n$, is an isomorphism from $F(n)$ to $F[n]$.
Indeed, since $\psi'$ is an isomorphism from $F(n-1)$ to $F[n]$, we have that $\psi'\Gamma=\Gamma\psi'$ and $\psi'\Gamma^+=\Gamma^+\psi'$, where $\Gamma$ and $\Gamma^+$
denote undirected and directed adjacency, respectively. Thus, from
\begin{align*}
 \psi'\Gamma(a|x1\ldots x_{n-1})&=\psi'(-a|x_1\ldots x_{n-1})=\alpha(-a)|\pi'(x_1\ldots x_{n-1}),\\
 \Gamma\psi'(a|x_1\ldots x_{n-1})&=\Gamma(\alpha(a)|\pi'(x_1\ldots x_{n-1})=\overline{\alpha(a)}|\pi'(x_1\ldots x_{n-1}),
\end{align*}
and
\begin{align*}
 \psi'\Gamma^+(a|x_1\ldots x_{n-1})&=\psi'(a|x_2\ldots x_{n-1}x_{n-1}+a)
  =\alpha(a)|\pi'(x_2\ldots x_{n-1}x_{n-1}+a),\\
 \Gamma^+\psi'(a|x_1\ldots x_{n-1})&=\Gamma^+(\alpha(a)|\pi'(x_1\ldots x_{n-1}))
 =\alpha(a)|-2\cdot\pi'(x_1\ldots x_{n-1})+\alpha(a), 
\end{align*}
we conclude that $\alpha(-a)=\overline{\alpha(a)}$ for every $a\in\{+1,-1\}$ (as it is immediate to check), and
\begin{equation}
\label{eq-iso-F(n-1)}
\pi'(x_2\ldots x_{n-1}(x_{n-1}+a))=-2\cdot\pi'(x_1\ldots x_{n-1})+\alpha(a).
\end{equation}
Now, we can assume that $a=+1$ (because of the automorphism \eqref{isomorfF(n)}), and let $a'=x_n-x_{n-1}$. Then, since clearly $\psi\Gamma=\Gamma\psi$, edges map to edges, we focus on proving that the same holds for the arcs, that is, $\psi^+\Gamma=\Gamma\psi^+$. With this aim, we need to prove that the following two calculations, where we use \eqref{psi-F(n)}, give the same result:
\begin{align}
 \psi\Gamma^+(+1|x_1\ldots x_{n})&=\psi(+1|x_2\ldots x_{n}x_{n}+1)\nonumber\\
  &=2|-2\cdot\pi'(x_2\ldots x_{n})+2,\\
 \Gamma^+\psi(+1|x_1\ldots x_{n})&=\Gamma^+(-2\cdot\pi'(x_1\ldots x_{n-1})+\alpha(a')\nonumber\\
 &=2|4\cdot\pi'(x_1\ldots x_{n-1})-2\alpha(a')+2.
\end{align}
The required equality follows since, from \eqref{eq-iso-F(n-1)} with $a'$ instead of $a$, we have
\begin{align*}
-2\cdot\pi'(x_2\ldots x_{n})&=2\cdot\pi'(x_2\ldots x_{n-1}(x_{n-1}+a'))=-2[-2\cdot\pi'(x_1\ldots x_{n-1})+\alpha(a')]\\
 &=4\cdot\pi'(x_1\ldots x_{n-1})-2\alpha(a').
\end{align*}
In Figure \ref{fig:F(3)}, every vertex has been labeled according to both presentations.\\
Using this presentation, we extend (and again prove) Proposition \ref{diam-F(n)}.
 \begin{proposition}
The diameter of $F(n)$ is $k=2n$. More precisely, there is a path of length $n$ or $n-1$ between any pair of  edges $\alpha|i-\overline{\alpha}|i$ and $\alpha'|i'-\overline{\alpha'}|i'$. Moreover, there is a path of length between $n-1$ and $2n$ between any pair of vertices.
\end{proposition}

\begin{proof}
Let us consider a tree rooted at a pair of vertices of an edge, $\vecu_1=1|i$ and $\vecu_2=2|i$, and suppose the $n=2r+1$ is odd (the case of even $n$ is similar). Then,
\begin{itemize}
    \item 
    The vertices at distances $1,2$ of $\vecu_1$ or $\vecu_2$ are 
    $\alpha|-2i+1$, $\alpha|-2i+2$ with $\alpha=1,2$.
   \item 
    The vertices at distances $3,4$ of $\vecu_1$ or $\vecu_2$ are 
    $\alpha|4i$, $\alpha|4i-1$, $\alpha|4i-2$ and $\alpha|4i-3$ with $\alpha=1,2$. 
    \item 
    The vertices at distances $5,6$ of $\vecu_1$ or $\vecu_2$ are 
    $\alpha|-8i+1$, $\alpha|-8i+2$, \ldots, $\alpha|-8i+8$  with $\alpha=1,2$.
    \item[]\qquad \vdots
    \item 
    The vertices at distances $2n-3,2n-2$ of $\vecu_1$ or $\vecu_2$ are 
    $\alpha|2^{n-1}+r$ with $r=0,-1,\ldots,-2^{n-1}+1$ and $\alpha=1,2$.
     \item 
    The vertices at distances $2n-1,2n$ of $\vecu_1$ or $\vecu_2$ are
   $\alpha|-2^n+r$ with $r=1,2,\ldots,2^{n}$ and $\alpha=1,2$.
\end{itemize}
See Figure \ref{fig:paths} for the case of $F(3)$, which has $24$ vertices. Note that, from the pair of vertices $1|i$ and $2|i$, the $3$-rd and $4$-th columns contain all the `consecutive' vertices of $F(3)$ from $\alpha|4i-3$ to $a|4i+8$, with $\alpha=1,2$. More precisely, from vertex $2|i$ (we can fix $\alpha$ because of the automorphism), we reach all of such vertices with at most $6$ steps, except $2|4i+1$ (in boldface, on the top of the $4$-th column), which would require the $7$ adjacencies 
`$-\rightarrow-\rightarrow-\rightarrow-$'. But this vertex is reached following the path `$\rightarrow\rightarrow\rightarrow-\rightarrow-$' (in boldface, in the  $5$-th column). In general, using the notation $f(\alpha|i)=\alpha|-2i+\alpha$
and $g(\alpha|i)=\overline{\alpha}|i$, we have the following: Let $N'=3\cdot2^{n-1}$. Then,
\begin{itemize}
\item
If $n$ is even, then the exception vertex is $g(fg)^n(2|i) (\textrm{mod}\ N) =1|2^n i$ ($2n+1$ steps) but $(gf)^{n-1}f^2(2|i) (\textrm{mod}\ N) =1|2^n i$ ($2n$ steps).
\item
If $n$ is odd, then the exception vertex is $g(fg)^n(2|i)(\textrm{mod}\ N) =2|2^{n-1} i+1$ ($2n+1$ steps) but $(gf)^{n-1}f^2(2|i) (\textrm{mod}\ N) =2|2^{n-1} i+1$ ($2n$ steps).
\end{itemize}
\end{proof}

\begin{figure}[t]
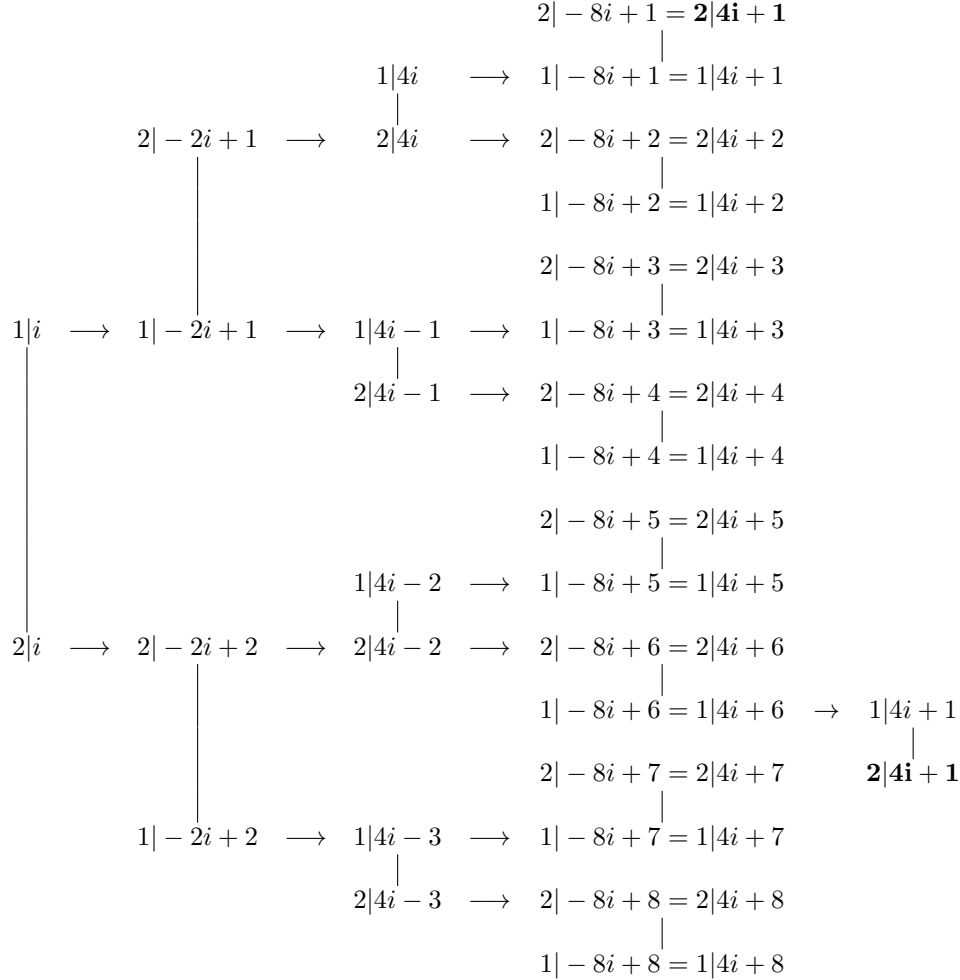

\begin{footnotesize}
\begin{center}
    $$
\begin{array}{ccccccccc}
    &             &         &              &        &             &  2|-8i+1={\bf 2|4i+1} & & \\
    &             &         &              &        &             &  \vline  & & \\  
    &             &         &              & 1|4i   & \longrightarrow &  1|-8i+1=1|4i+1  & &\\
    &             &         &              & \vline &             &          & & \\  
    &             & 2|-2i+1 & \longrightarrow  & 2|4i   & \longrightarrow &  2|-8i+2=2|4i+2  & &\\
    &             & \vline  &              &        &             &  \vline  & & \\ 
    &             & \vline  &              &        &             &  1|-8i+2=1|4i+2  & &\\
    &             & \vline  &              &        &             &          & & \\
    &             & \vline  &              &        &             &  2|-8i+3=2|4i+3  & &\\
    &             & \vline  &              &        &             &  \vline  & & \\ 
1|i & \longrightarrow & 1|-2i+1 & \longrightarrow  & 1|4i-1 & \longrightarrow &  1|-8i+3=1|4i+3  & &\\
\vline    &             &         &              & \vline &             &         & & \\
\vline    &             &         &              & 2|4i-1 & \longrightarrow &  2|-8i+4=2|4i+4  & & \\
\vline    &             &         &              &        &             &  \vline  & & \\ 
\vline    &             &         &              &        &             &  1|-8i+4=1|4i+4  & & \\
\vline    &             &         &              &        &             &          & & \\ 
\vline    &             &         &              &        &             &  2|-8i+5=2|4i+5  & & \\
\vline    &             &         &              &        &             &  \vline  & & \\ 
\vline   &             &         &              & 1|4i-2 & \longrightarrow &  1|-8i+5=1|4i+5  & & \\
\vline    &             &         &              & \vline &             &          & & \\
2|i & \longrightarrow & 2|-2i+2 & \longrightarrow  & 2|4i-2 & \longrightarrow &  2|-8i+6=2|4i+6  & & \\
    &             & \vline  &              &        &             &  \vline   & &\\ 
    &             & \vline  &              &        &             &  1|-8i+6=1|4i+6  & \rightarrow & 1|4i+1\\
    &             & \vline  &              &        &             &          & & \vline\\
    &             & \vline  &              &        &             &  2|-8i+7=2|4i+7  & & {\bf 2|4i+1}\\
    &             & \vline  &              &        &             &  \vline   & &\\ 
    &             & 1|-2i+2 & \longrightarrow  & 1|4i-3 & \longrightarrow &  1|-8i+7=1|4i+7  & & \\
    &             &         &              & \vline &             &          & & \\
    &             &         &              & 2|4i-3 & \longrightarrow &  2|-8i+8=2|4i+8  & & \\
    &             &         &              &        &             &  \vline  & & \\ 
    &             &         &              &        &             &  1|-8i+8=1|4i+8  & &
\end{array}
$$
\end{center}
	\caption{The paths of length at most $6$ in $F(3)$ from the vertices of the edge $\{1|i,2|i\}$.}
	\label{fig:paths}
 \end{footnotesize}
\end{figure}

\subsection{The mixed graphs $F^*(n)$}

A variation of the mixed graphs $F(n)$ allows us to obtain $(1,1,k)$-regular mixed graphs that we denote $F^*(n)$.
Given $n\ge 2$, the $(1,1,k)$-regular mixed graph $F^*(n)$ has vertices labeled as those of $F(n)$. That is, 
$a|x_1\ldots x_n$, where $a\in\{+1,-1\}$ and $x_i\in \mathbb{Z}_3$. Now the adjacencies are as follows:
\begin{align}
a|x_1x_2\ldots x_n\,& \sim\, -a|x_1x_2\ldots x_n \mbox{\ (edges);}\label{edges-1}\\
a|x_1x_2\ldots x_n\,&\rightarrow\, a|x_2x_3\ldots x_n(x_n+a(x_2-x_1)) \mbox{\ (arcs)}\label{arcs-1},
\end{align}
where, when computed modulo 3,  we take $x_2-x_1\in \{+1,-1\}$. Hence, the vertices $a|x_1x_2\ldots x_n$ and $a|x'_1x_2\ldots x_n$, with $x_1'\neq x_1$, are  adjacent to different vertices of the form $a|x_2\ldots (x_n\pm 1)$.

For example, the mixed graph $F^*(3)$ is shown in Figure \ref{fig:Fstar(3)}.

 \begin{figure}[t]
    \begin{center}
        \includegraphics[width=8cm]{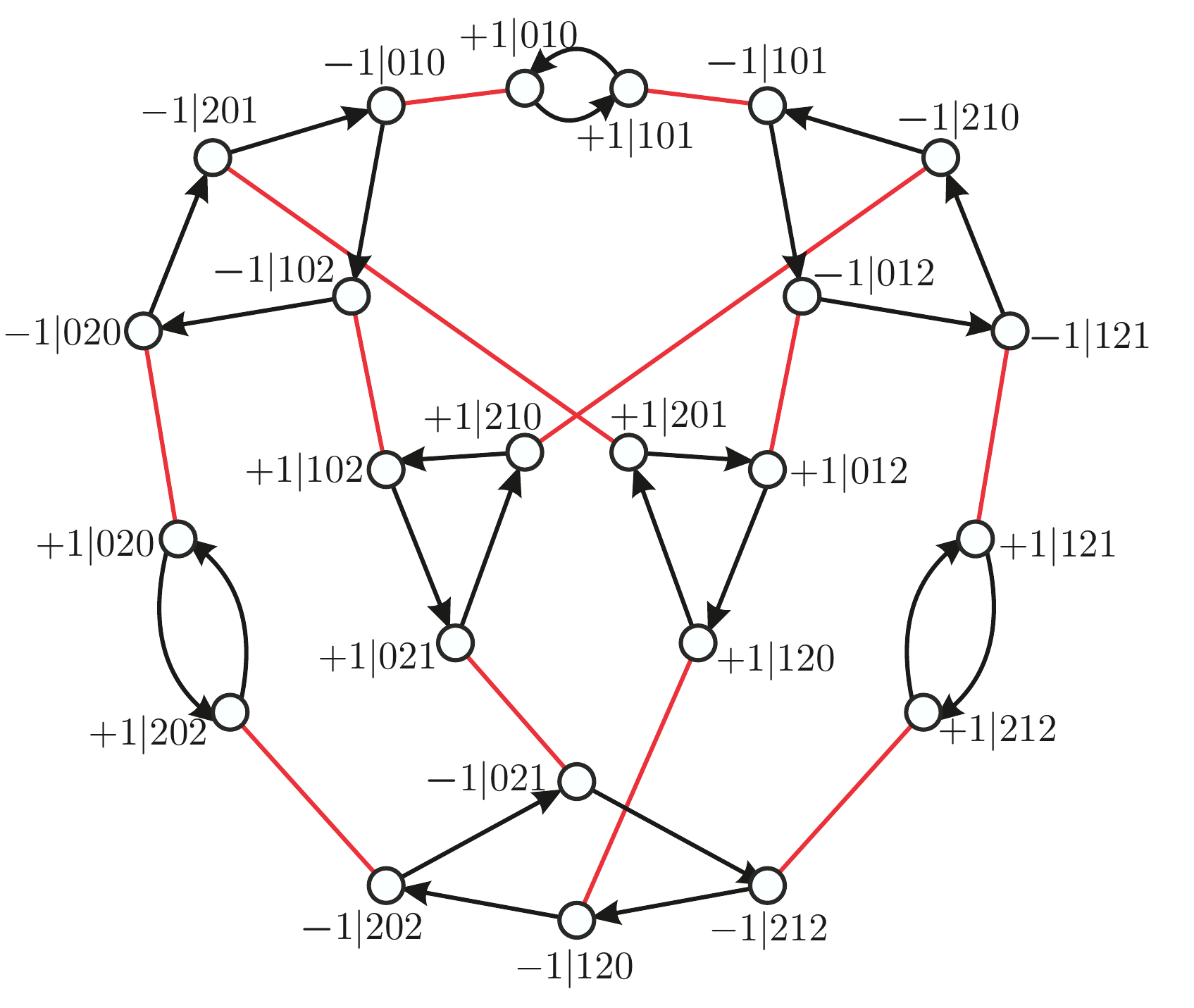}
    \end{center}
    \vskip-.5cm
	\caption{The mixed graph $F^*(3)$.}
	\label{fig:Fstar(3)}
\end{figure}

\subsubsection{An alternative presentation}

To study some properties of $F^*(n)$, it is useful to work with the following equivalent presentation:
The vertices are now labeled as $a|b:a_1\ldots a_{n-1}$, where $a,a_i\in\{+1,-1\}$ for $i=1,\ldots,n-1$, and $b\in \mathbb{Z}_3$. Then, the adjacencies \eqref{edges-1} and \eqref{arcs-1} become 
\begin{align}
a|b:a_1a_2\ldots a_{n-1}\,& \sim\, -a|b:a_1a_2\ldots a_{n-1} \mbox{\  (edges);}\\
a|b:a_1a_2\ldots a_{n-1}\,&\rightarrow\, a|b+a_1:a_2a_3\ldots a_{n-1}\,aa_1 \mbox{\  (arcs).}
\end{align}
Notice that a vertex $a|x_1x_2\ldots x_n$ with the old presentation is now labeled as $a|b:a_1\ldots a_{n-1}$ with $b=x_1$ and $a_i=x_{i+1}-x_i$ for $i=1,\ldots,n-1$. From this, it is readily checked that the `new' adjacencies
are as mentioned.
\begin{proposition}
The group of automorphisms of $F^*(n)$ is isomorphic to the dihedral group $D_3$.
\end{proposition}
\begin{proof}
Using the new notation, let us first show that the following mappings, $\Phi$ and $\Psi$, are automorphisms of $F^*(n)$:
\begin{align}
\Phi(a|b:a_1a_2\ldots a_{n-1})& =a|\phi(b):\overline{a_1}\,\overline{a_2}\ldots \overline{a_{n-1}};\label{Phi}\\
\Psi(a|b:a_1a_2\ldots a_{n-1})& =a|b+1:a_1a_2\ldots a_{n-1},\label{Psi}
\end{align}
where $\phi(0)=1$, $\phi(1)=0$, $\phi(2)=2$, and $\overline{a_i}=-a_i$ for $i=1,\ldots,n-1$. To prove that $\Phi$ is an automorphism of $F^*(n)$, observe that the vertex in \eqref{Phi} is adjacent, through an edge, to
$$
\overline{a}|\phi(b):\overline{a_1}\,\overline{a_2}\ldots \overline{a_{n-1}}
=\Phi(\overline{a}|b:a_1a_2\ldots a_{n-1}),
$$
and, through an arc, to 
$$
a|\phi(b)+\overline{a_1}:\overline{a_2}\ldots \overline{a_{n-1}}\,a\overline{a_1}
=\Phi( a|b+a_1:a_2a_3\ldots a_{n-1}\,aa_1),
$$
where the last equality holds since $\phi(b+a_1)=\phi(b)+\overline{a_1}$, and $a\overline{a_1}=\overline{aa_1}$ for every $b\in\mathbb{Z}_3$ and $a,a_1\in\{+1,-1\}$.
Similarly, we can prove that $\Psi$ is also an automorphism of $F^*(n)$.
Clearly, $\Phi$ is involutive, and $\Psi$ has order three. Moreover, $(\Phi\Psi)^2= \id$ (the identity). Then, the automorphism group $\Aut(F^*(n))$ must contain the subgroup $\langle \Phi,\Psi\rangle=D_3$. It is easy to see that the graph $F^*(n)$  has exactly three digons between pairs of vertices of the form $-1:xyxy\ldots xy$ and $-1:yxyx\ldots yx$  when $n$ is even, or $+1:xyxy\ldots x$ and $+1:yxyx\ldots y$ when $n$ is odd; see again Figure \ref{fig:Fstar(3)}. Thus,
any automorphism of $F^*(n)$ must interchange these digons; hence, the automorphism group has at most $3!=6$ elements. Consequently, $\Aut(F^*(n))\cong D_3\cong S_3$, as claimed.
\end{proof}

Before giving the diameter of $F^*(n)$, we show that, for every vertex $\vecu$, there is only a possible vertex $\vecv$ at distance $2n+1$ from $\vecu$. 
Suppose first that $n$ is even (the case of odd $n$ is similar). It is clear that, excepting possibly one case, from vertex $\vecu=a|b:a_1a_2\ldots a_{n-1}$ to vertex $\vecv=a'|b':y_1y_2\ldots y_{n-1}$, there is a path with at most $2n$ steps of the form $- \rightarrow - \rightarrow \cdots - \rightarrow$, where `$-$' stands for `$\sim$' (edge) or `$\emptyset$' (nothing), and `$\rightarrow$' represents an arc.  The exception occurs when all the edges of the path are necessary. That is:
 \begin{itemize}
 \item 
 If $b'=b+\Sigma +\overline{a}$ (where $\Sigma=\textstyle\sum_{i=1}^{n-1} a_i$, and so that  $b'\neq b+\Sigma$), then the first two steps are
 $$
 a|b:a_1a_2\ldots a_{n-1}\ \sim \ \overline{a}|b:a_1a_2\ldots a_{n-1}\ \rightarrow \ \overline{a}|b+a_1:a_2a_3\ldots a_{n-1}\,\overline{a}a_1.
 $$
 \item 
 If $y_1=aa_2$, then the next two steps are
 \begin{align*}
 \overline{a}|b+a_1:a_2a_3\ldots a_{n-1}\,\overline{a}a_1\ &\sim \ a|b+a_1:a_2a_3\ldots a_{n-1}\,\overline{a}a_1\ \\
  & \rightarrow \  a|b+a_1+a_2:a_3\ldots a_{n-1}\,\overline{a}a_1\, aa_2.
 \end{align*}
 \item 
 If $y_1=\overline{a}a_3$, then the next two steps are
 \begin{align*}
 a|b+a_1+a_2:a_3\ldots a_{n-1}\,\overline{a}a_1\, aa_2\ &\sim \ \overline{a}|b+a_1+a_2:a_3\ldots a_{n-1}\,\overline{a}a_1\, aa_2\ \\
  & \rightarrow \  \overline{a}|b+a_1+a_2+a_3:a_4\ldots a_{n-1}\,\overline{a}a_1\, aa_2\, \overline{a}a_3.
 \end{align*}
 \item[] $\vdots$
 \item 
 If $y_{n-1}=a\overline{a}a_1=-a_1$, then the last two steps are
 \begin{align*}
 \overline{a}|b+\Sigma:\overline{a}a_1\,aa_2\,\overline{a}a_3\ldots \overline{a}a_{n-1}\, &\sim \ a|b+\Sigma:\overline{a}a_1\,aa_2\ldots \overline{a}a_{n-1}\ \\
  & \rightarrow \  a|b+\Sigma+\overline{a}a_1:aa_2\ldots \overline{a}a_{n-1}\,\overline{a_1}\\
  &= a|b':y_1y_2\ldots y_{n-1}.
 \end{align*}
 \end{itemize}
 Thus, if $a\neq a'$ $(a'=\overline{a})$, the vertex 
 $$
 \vecv=\overline{a}|b+\Sigma+\overline{a}a_1:aa_2\,\overline{a}a_3\ldots \overline{a}a_{n-1}\,\overline{a_1}
 $$
 is not reached from $\vecu$ in this way. Similarly, if $n$ is odd,
 the exception is the vertex
 $$
 \vecv=a|b+\Sigma+\overline{a}a_1:aa_2\,\overline{a}a_3\ldots aa_{n-1}\,\overline{a_1}.
 $$

\subsection{The mixed graphs $F'(n)$}

If necessary, the  three digons of $F^*(n)$ can be removed and replaced by three new edges of the form 
\begin{align*}
+1:xyxy\ldots xy &\ \sim\  +1:yxyx\ldots yx\qquad \mbox{($n$ even)},\\
-1:xyxy\ldots yx  &\ \sim\ -1:yxyx\ldots xy\qquad \mbox{($n$ odd)}.
\end{align*}
So, we obtain the new mixed graph $F'(n)$, with  $N=3\cdot 2^n-6$ vertices and diameter $k\le 2n$. More precisely, $F'(2)$ is isomorphic to the Kautz digraph $K(2,2)$ with $N=6$ vertices and diameter $k=2$; and when $n\in\{3,4\}$, the mixed graph $F'(n)$ has diameter $k=2n-1$. For instance, the mixed graph $F'(4)$, with $N=42$ vertices and diameter $k=7$, is shown in Figure 	\ref{fig:F'(4)}.
In all the other cases, when $n\ge 5$, computational results seem to show that the diameter of $F'(n)$ is always $k=2n$.

 \begin{figure}[t]
    \begin{center}
        \includegraphics[width=12cm]{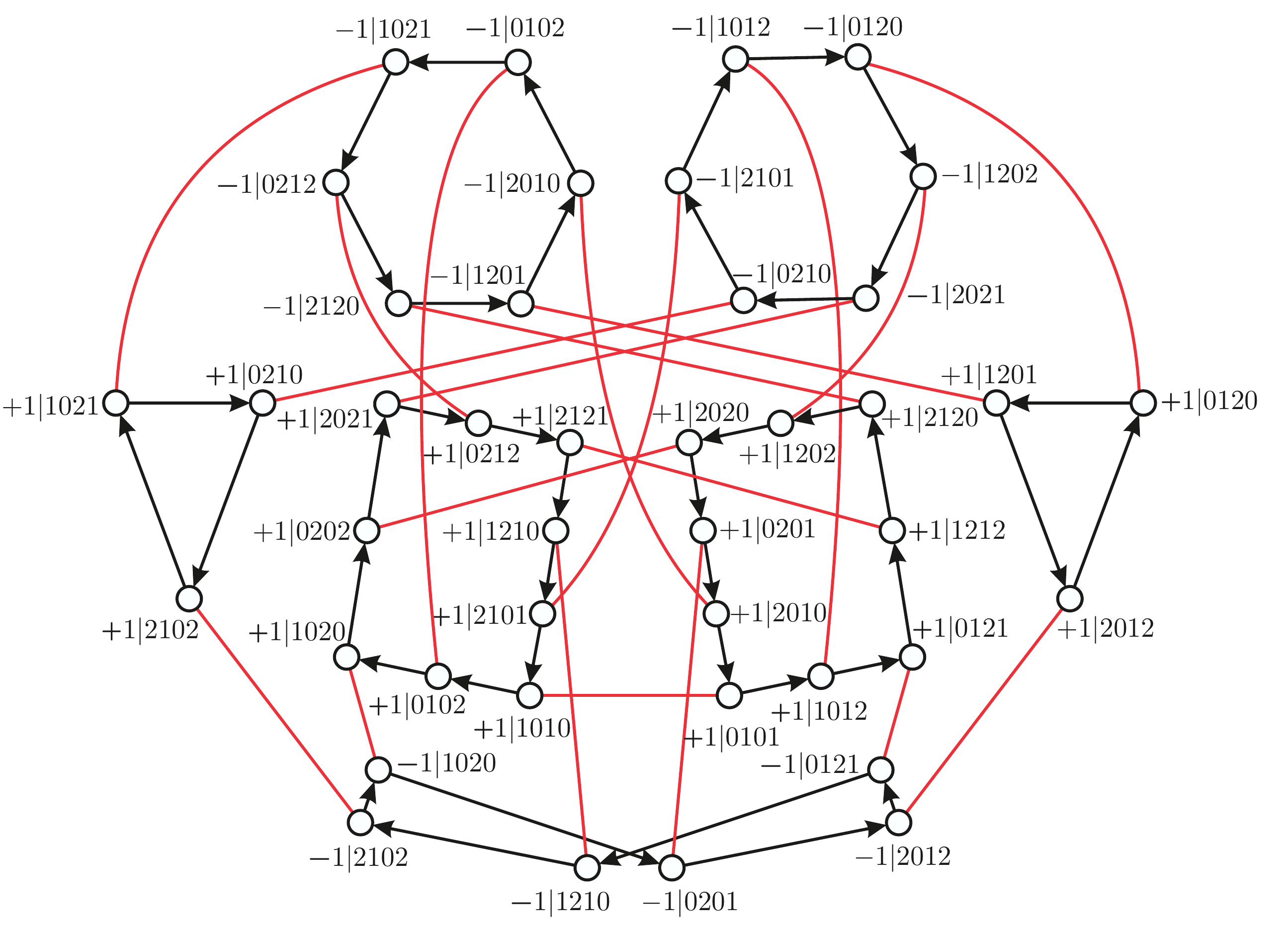}
    \end{center}
    \vskip-.5cm
	\caption{The mixed graph $F'(4)$ with $42$ vertices and diameter $7$.}
	\label{fig:F'(4)}
\end{figure}

\subsection{The mixed graphs $G(n)$}

We define a $(1,1,k)$-regular mixed graph $G(n)$, for $n\ge 2$, as follows: the vertices are of the form $x_0|x_1\ldots x_n$, where $x_i\in\mathbb{Z}_2$ for $i=0,1,\ldots,n$. More precisely, the vertices are:
\begin{itemize}
	\item[$\circ$] For any $n$: $1|00\ldots0$ and  $1|11\ldots1$;
	\item[$\circ$] For odd $n$: $0|0101\ldots0$ and $0|1010\ldots1$;
	\item[$\circ$] For even $n$: $1|0101\ldots01$ and $1|1010\ldots10$;
	\item[$\circ$] For the other vertices, $0|x_1\ldots x_n$ and $1|x_1\ldots x_n$, with $x_i\in\mathbb{Z}_2$.
\end{itemize}
So, the number of vertices of $G(n)$ is $2^{n+1}-4$.\\
The adjacencies (with arithmetic modulo 2) through edges are:
\begin{itemize}
	\item[$(i)$] For any $n$: $1|00\ldots0 \,\sim\, 1|11\ldots1$;
	\item[$(ii)$] For odd $n$: $1|0101\ldots0 \,\sim\, 1|1010\ldots1$;
	\item[$(iii)$] For even $n$: $0|0101\ldots01 \,\sim\, 0|1010\ldots10$;
	\item[$(iv)$] For the other vertices, $x_0|x_1\ldots x_n \,\sim\, (x_0+1)|x_1\ldots x_n$.
\end{itemize}
The adjacencies through arcs are:
\begin{itemize}
	\item[$(v)$] $x_0|x_1\ldots x_n$ $\rightarrow$ $x_0|x_2\ldots x_n (x_1+x_0)$.
\end{itemize}
The graph $G(n)$ is an in- and out-regular mixed graph with $r=z=1$.
Its only nontrivial automorphism is the one that sends $x_0|\x=x_0|x_1 x_2 x_3 \ldots$ to $x_0|\overline{\x}=x_0 |\overline{x_1}\, \overline{x_2}\, \overline{x_3} \ldots$, where $\overline{x_i}=x_i+1$ for $i=1,2,3,\ldots$
In Figure \ref{fig:G(3)}, we show the mixed graph $G(3)$.

 \begin{figure}[t]
    \begin{center}
        \includegraphics[width=10cm]{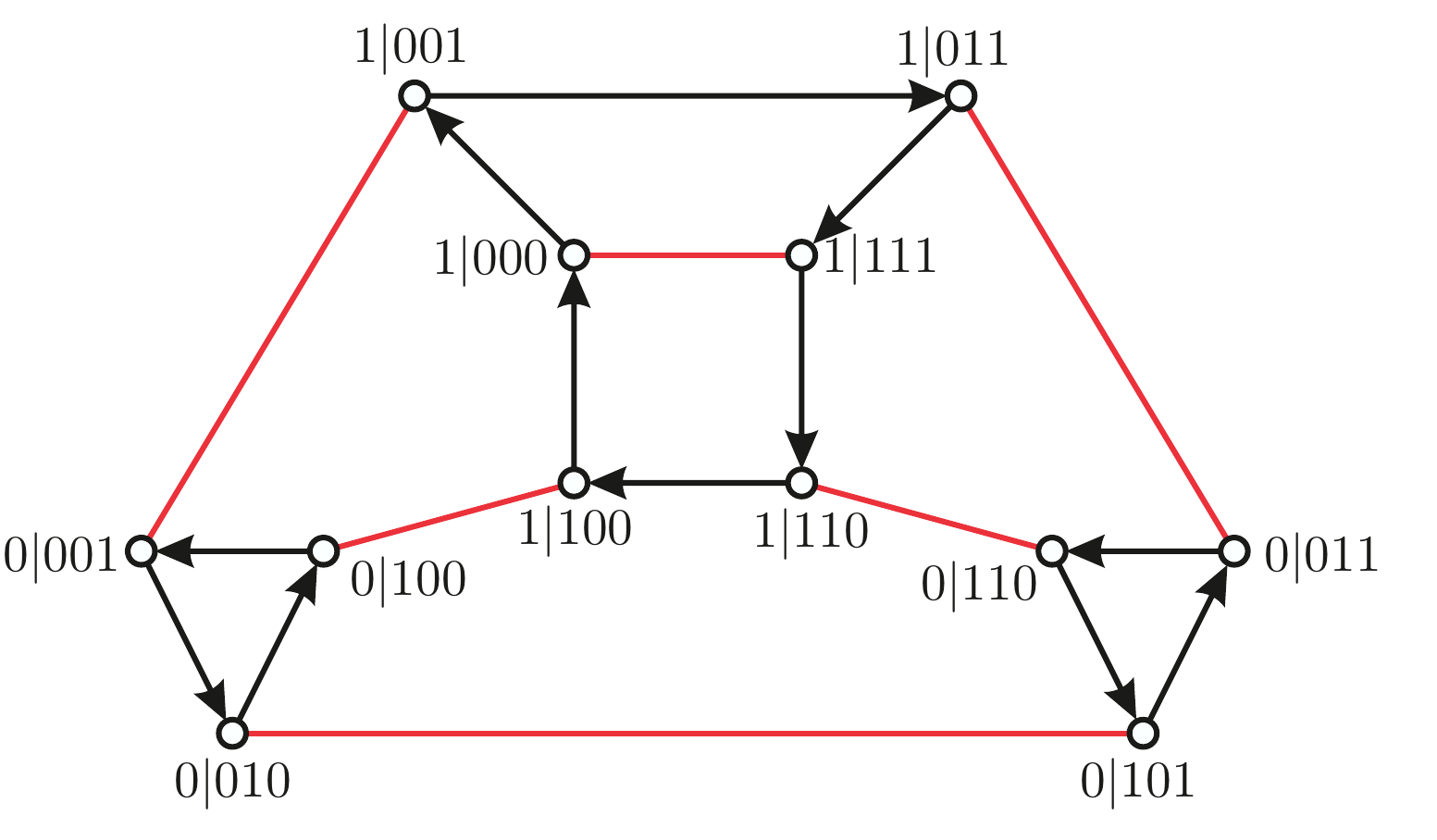}
    \end{center}
    \vskip-.5cm
	\caption{The mixed graph $G(3)$.}
	\label{fig:G(3)}
\end{figure}

Looking at the results for $n\le 12$ obtained by computer, we are led to conjecture that the diameter of $G(n)$ is $k=2n-1$. At first sight, the proof of this result seems to be involved, although we managed to prove the following.

\begin{proposition}
The diameter of $G(n)$ is at most $2n$.
\end{proposition}
\begin{proof}
Consider the digraph $G^+(n)$ defined by considering {\bf all} $2^{n+1}$ vertices of the form $0|x_1\ldots x_n$ and $1|x_1\ldots x_n$, with $x_i\in\mathbb{Z}_2$, with undirected adjacencies as in $(iv)$, and directed adjacencies as in $(v)$.
Then, $G^+(n)$ has the self-loops at vertices $0|00\ldots 0$ and $0|11\ldots 1$ and one digon
(or two opposite arcs) between $0|0101\ldots 01$ and $0|1010\ldots 10$ for even $n$,
and $0|0101\ldots 0$ and $0|1010\ldots 1$ for odd $n$.
In fact, if every edge of $G^+(n)$ is `contracted' to a vertex, what remains is the De Bruijn digraph $B(2,n)$, with $2^n$ vertices and diameter $n$.
Moreover, notice that $G(n)$ is obtained by removing the above four vertices and adding the edges in $(i)$, $(ii)$, and $(iii)$.
By way of examples, Figure \ref{fig:G+(2)} shows the graph $G^+(2)$, whereas Figure \ref{fig:excG+(3)} shows the mixed graph $G^+(3)$ `hanging' from a vertex with eccentricity $2n=6$.

Consequently, since the diameter of $G(n)$ is upper bounded by the diameter of $G^+(n)$, we  concentrate on proving that the diameter of $G^+(n)$ is $2n$ for $n>1$ ($G(1)$ has diameter $3$).
The proof is constructive because we show a walk of length at most $2n$ between any pair of vertices. To this end, we take the following steps:

\begin{enumerate}
\item
There is a walk of length at most $2n$ from vertex $x_0|\x=x_0|x_1x_2\ldots x_n$ to vertex
$(x_n+y_n)|y_1y_2\ldots y_n$. Indeed, as $x_i+x_i=0$ for any value of $x_i$, we get
\begin{align}
x_0|x_1x_2x_3\ldots x_n & \,\sim\, (x_1+y_1)|x_1x_2x_3\ldots x_n \rightarrow (x_1+y_1)|x_2x_3\ldots x_n y_1 \nonumber\\
                     & \,\sim\, (x_2+y_2)|x_2x_3\ldots x_ny_1 \rightarrow (x_2+y_2)|x_3\ldots x_n y_1 y_2 \nonumber\\
                     & \ \vdots \label{walk}\\
                     & \,\sim\, (x_n+y_n)|x_ny_1y_2y_3\ldots y_{n-1} \rightarrow (x_n+y_n)|y_1y_2\ldots y_n.\nonumber
\end{align}
Thus,  the initial vertex $x_0|x$ and the {\em step pattern} `$\sim\rightarrow\sim\rightarrow\stackrel{(2n)}{\cdots\cdots} \sim\rightarrow$' uniquely determine the destiny vertex.
\begin{figure}[t]
    \begin{center}
        \includegraphics[width=10cm]{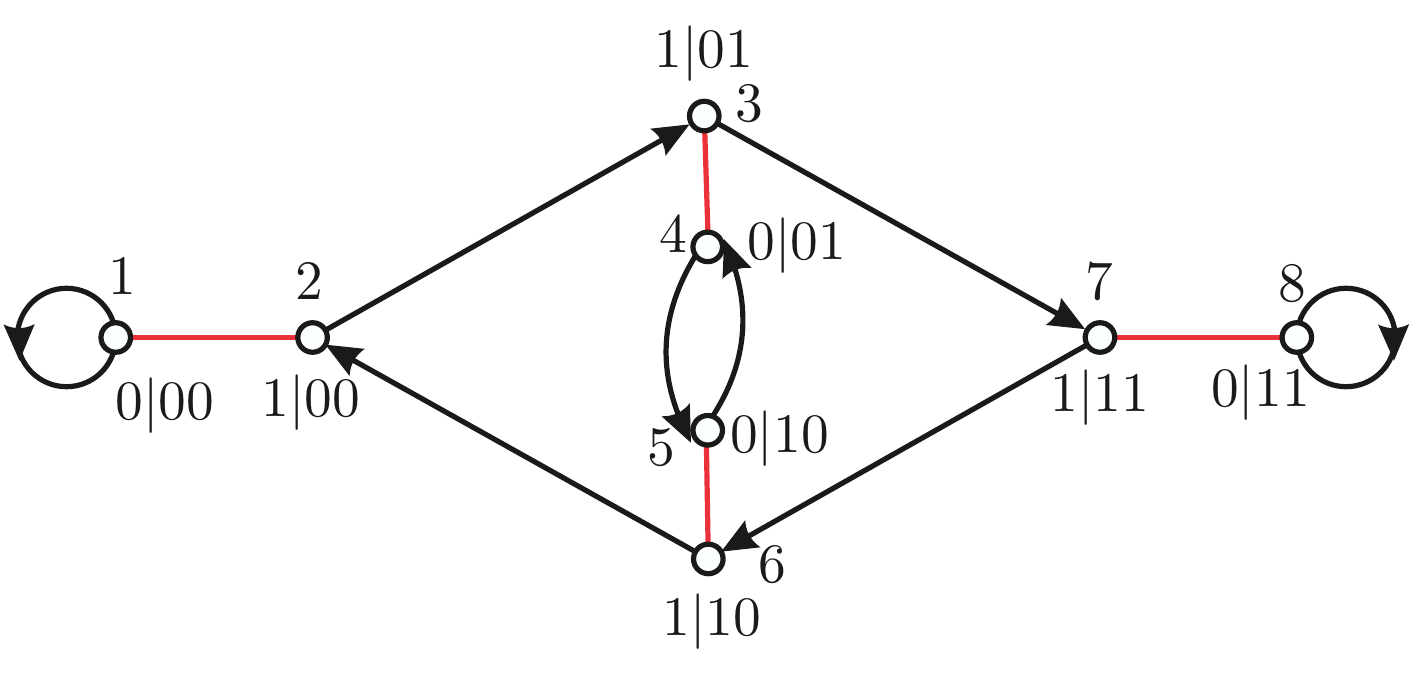}
    \end{center}
    \vskip-.5cm
	\caption{The graph $G^+(2)$.}
	\label{fig:G+(2)}
\end{figure}
 \begin{figure}[t]
	\begin{center}
		\includegraphics[width=12cm]{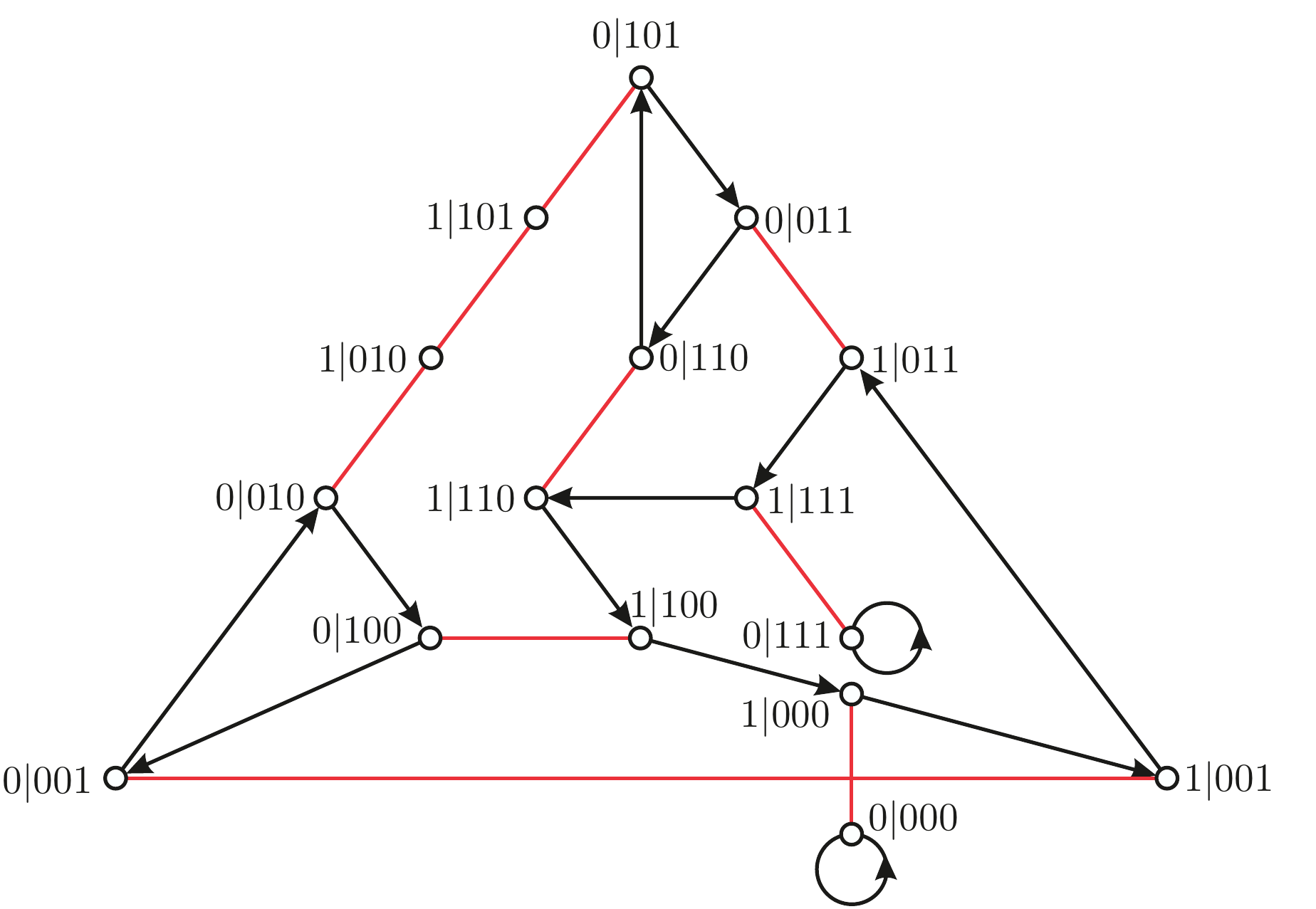}
	\end{center}
	\vskip-.5cm
	\caption{The graph $G^+(3)$ `hanging' from vertex $0|101$, with eccentricity $6$. 
	}
	\label{fig:excG+(3)}
\end{figure}
\item
Clearly, some of the steps in \eqref{walk} are {\bf not} necessary if some of the following situations occur:
\begin{enumerate}
\item
The `intersection' of the sequences $\x=x_1x_2\ldots x_n$ and $\y=y_1y_2\ldots y_n$ (that is, the maximum length of the last subsequence of $\x$ that coincides with a first subsequence of $\y$), denoted $|\x\cap\y|$, is greater than zero. (For instance, for
$\x=0\ldots 010$ and $\y=100\ldots 0$, we get $|\x\cap\y|=2$.)
In this case, the first $\ell=|\x\cap\y|$  step pairs `$\sim\rightarrow$' of the walk in \eqref{walk} are useless and can be avoided. Then, we say that we {\em save} $2\ell$ steps.
\item
Some of the following equalities hold:
$x_0=x_1+y_1$, or $x_i+y_i=x_{i+1}+y_{i+1}$ for some $i=1,\ldots,n-1$. In this case, some steps `$\sim$' are absent. More precisely, if either both equalities $x_i=y_i$ and $x_{i+1}=y_{i+1}$ (or both inequalities $x_i\neq y_i$ and $x_{i+1}\neq y_{i+1}$) hold, then the step `$\sim$' through an edge leading to $(x_{i+1}+y_{i+1})|x_{i+1}\ldots x_ny_1\ldots y_{i}\ldots$ is absent. So, we save $1$ step.
\end{enumerate}
Thus, if we can save some steps, one last step $(x_n+y_n)|\y \sim (\overline{x_n+y_n})|\y$ assures a walk of length at most $2n$ from $x_0|\x$ to $y_0|\y$ for any $y_0\in\{0,1\}$.
\item
In the `worst case', the walk in \eqref{walk} consists of exactly $2n$ steps (vertices at maximum distance) if $|\x\cap\y|=0$ and none of the equalities in (b) holds.
Assuming first that $x_0=0$ (the case $x_0=1$ is similar), the latter occurs when $x_1+y_1=1\Rightarrow y_1=\overline{x_1}$, $x_2+y_2=0\Rightarrow y_2=x_2$, $x_3+y_3=1\Rightarrow y_3=\overline{x_3}$, and so on.
Consequently, starting from $0|\x=0|x_1x_2x_3\ldots x_n$, we only need to test the destiny vertices of the form $1|\y=1|\overline{x_1} x_2 \overline{x_3}\ldots x_n$ ($n$ even),
and $0|\z=$ $0|\overline{x_1} x_2 \overline{x_3}\ldots \overline{x_n}$ ($n$ odd), with the additional constraints $|\x\cap\y|=|\x\cap\z|=0$.
\item
For these cases, the strategy is to put first the last digit of destiny. Namely, if $n$ is even,
\begin{align}
0|x_1x_2x_3\ldots x_n & \, \rightarrow 0|x_2x_3\ldots x_n x_1 \nonumber\\
                     & \,\sim\, (x_2+\overline{x_1})|x_2x_3\ldots x_n x_1 \rightarrow (x_2+\overline{x_1})|x_3x_4\ldots x_n x_1 \overline{x_1}\, \nonumber\\
                      & \,\sim\,(x_3+x_2)|x_3x_4\ldots x_n x_1 \overline{x_1}\,  \rightarrow
                     (x_3+x_2)|x_4x_5\ldots x_n x_1 \overline{x_1}\, x_2\nonumber\\
                     & \,\sim\,(x_4+\overline{x_3})|x_4 x_5\ldots x_n x_1 \overline{x_1}\,x_2  \rightarrow
                     (x_4+\overline{x_3})|x_5\ldots x_n x_1 \overline{x_1}\, x_2 \overline{x_3}\nonumber\\
                     & \ \vdots \label{walk2}\\
                     & \,\sim\, (x_n+\overline{x_{n-1}})|x_n x_1\overline{x_1}x_2\ldots \overline{x_{n-1}} x_{n-2}\,\nonumber\\
                     & \,\rightarrow\, (x_n+\overline{x_{n-1}})| x_1\overline{x_1}x_2\ldots x_{n-2}\overline{x_{n-1}}\nonumber\\
                      & \, \sim\, (x_1+x_n)|x_1\overline{x_1}x_2\ldots x_{n-2}\overline{x_{n-1}}\,
                       \rightarrow (x_1+x_n)|\overline{x_1}x_2\overline{x_3}\ldots x_{n}\nonumber\\
                    & \,\sim\, 1|(x_1+x_n)|\overline{x_1}x_2\overline{x_3}\ldots x_{n}.\nonumber
                     \end{align}
This walk can have $2n+2$ steps whenever all steps `$\sim$' through edges are present. This is the case when
$x_2+\overline{x_1}\neq 0$, $x_3+x_2\neq x_2+\overline{x_1}$, $x_4+\overline{x_3}\neq  x_3+x_2$,\ldots, $x_1+x_{n}\neq x_{n}+\overline{x_{n-1}}$, and $x_1+x_{n}\neq 1$.
In turn, this implies the $n+1$ equalities
\begin{align}
x_1 &=x_3,\ x_3=x_5,\ \ldots, \ x_{n-3}=x_{n-1},\ x_{n-1}=x_1,\label{equs1}\\
x_1 &=x_2,\ x_2=x_4,\ x_4=x_6,\ \ldots, \ x_{n-2}=x_{n},\ x_n=x_1. \label{equs2}
\end{align}
Note that these sequences of equalities form two cycles (with odd and even subscripts)  rooted at $x_1$.
Thus, the number of inequalities, if any, must be at least $2$. In this case, at least $2$ steps `$\sim$' are absent in \eqref{walk2}, and we have a walk of length at most $2n$ between the vertices considered.\\
Otherwise, if {\bf all} the equalities \eqref{equs1}--\eqref{equs2} hold,
 the initial vertex must be  $0|000\stackrel{(n)}{\ldots} 00$ (the first digit $x_1$ can be fixed to $0$ since the mixed graph has an automorphism that sends $x_0|x_1x_2\ldots x_n$ to  $x_0|\overline{x_1}\,\overline{x_2}\,\ldots \overline{x_n}$), and the destiny vertex is $0|1010\stackrel{(n)}{\ldots} 10$.
The same reasoning for $n$ odd leads that, in the worst case (walk in \eqref{walk2} of length $2n+2$), the initial vertex is $0|000\ldots 0$ and the final vertex $0|1010\ldots 1$.
In such cases, we have a particular walk of the desired length.
\item
There is a walk of length $2n$ from $0|000\ldots 0$ to $1|1010\ldots 10$ ($n$ even) or to
$0|1010\ldots 01$ ($n$ odd) by using the following step pattern
$$
\sim\,\rightarrow\,\rightarrow \sim\,\rightarrow\, \sim\,\rightarrow\,\stackrel{(2n)}{\cdots\cdots}
\sim\,\rightarrow\,\rightarrow.
$$
For instance, for $n=6$, we get
\begin{align}
0|000000 & \,\sim\, 1|000000\rightarrow 1|000001 \rightarrow 1|000011\nonumber\\
                     & \,\sim\, 0|000011 \rightarrow 0|000110 \nonumber\\
                     & \,\sim\, 1|000110 \rightarrow 1|001101 \nonumber\\
                     & \,\sim\, 0|001101 \rightarrow 0|011010 \nonumber\\
                     & \,\sim\, 1|011010 \rightarrow 1|110101 \rightarrow 1|101010,\nonumber
                     \end{align}
and, for $n=7$,
\begin{align}
0|0000000 & \,\sim\, 1|0000000 \rightarrow 1|0000001 \rightarrow 1|0000011\nonumber\\
                     & \,\sim\, 0|0000011 \rightarrow 0|0000110 \nonumber\\
                     & \,\sim\, 1|0000110 \rightarrow 1|0001101 \nonumber\\
                     & \,\sim\, 0|0001101 \rightarrow 0|0011010 \nonumber\\
                     & \,\sim\, 1|0011010 \rightarrow 1|0110101 \nonumber\\
                     & \,\sim\, 0|0110101 \rightarrow 0|1101010 \rightarrow 0|1010101.\nonumber
\end{align}
\item
The case $x=1$ is similar, and we only mention the main facts.
Now, the `worst case' ($2n$ steps) in the walk in \eqref{walk} ($2n$ steps) occurs when,
 starting from $1|\x=1|x_1x_2x_3\ldots x_n$, we want to reach the destiny vertices of the form $0|\y=0|x_1 \overline{x_2} x_3\overline{x_4}\ldots \overline{x_n}$ ($n$ even),
 or $1|\z=1|x_1 \overline{x_2} x_3\overline{x_4}\ldots x_n$ ($n$ odd), with the additional constraints $|\x\cap\y|=|\x\cap\z|=0$.
 Now, following the same strategy as in step 4 above, it turns out that for the case of $2n+2$ steps, the following conditions must hold (assuming $n$ odd, the even case is similar):
 \begin{align}
x_1 &=x_2,\ x_2=x_3,\ \ldots,\ x_{n-1}=x_{n},\ x_{n}\neq x_1,\label{equs3}
\end{align}
which are clearly incompatible, and at least there must be another inequality (the last one in \eqref{equs3} is forced since the final vertex has $x_0=1$).  Again,  at least $2$ steps `$\sim$' are absent in \eqref{walk2}, and we have a walk of length at most $2n$ between the vertices considered. For example, for $n=5$, and assuming that
$x_4\neq x_5$ and $x_1=0$, the walk of $10$ steps from $1|00001$ to $1|01011$ is:
\begin{align}
1|00001 & \,\rightarrow\, 1|00011\nonumber\\
                     & \,\sim\, 0|00011 \rightarrow 0|00110 \nonumber\\
                     & \,\sim\, 1|00110 \rightarrow 1|01101 \nonumber\\
                     & \,\sim\, 0|01101 \rightarrow 0|11010 \nonumber\\
                     & \,\rightarrow\, 0|10101 \rightarrow 0|01011 \sim 1|01011.\nonumber
\end{align}
\end{enumerate}
This completes the proof.
\end{proof}

In fact, we implicitly proved the following.
\begin{lemma}
For every $n>1$, the mixed graph $G^+(n)$ satisfies the following.
\begin{itemize}
\item[$(i)$]
The vertices $0|00\ldots 0$ and $0|11\ldots 1$ have maximum eccentricity $2n$.
\item[$(ii)$]
The vertices $1|00\ldots 0$ and $1|11\ldots 1$ have eccentricity $2n-1$.
\item[$(iii)$]
If $n\ge 5$, the vertices $1|00\ldots 01$ and $1|11\ldots 10$ have eccentricity $2n-2$.
\end{itemize}
\end{lemma}

\begin{proof}
$(i)$ and $(ii)$ follow from the previous reasoning. To prove $(iii)$, we only need to check the distance from
$1|00\ldots 01$ to $0|00\ldots 0$. A shortest path between these two vertices is
$1|00\ldots 01\sim 0|00\ldots 01\rightarrow 0|0\ldots 010\rightarrow \cdots\rightarrow 0|10\ldots 00 \sim
1|10\ldots 00\rightarrow 1|00\ldots 00\sim 0|00\ldots 0$ of length $n+3\le 2n-2$ if $n\ge 5$.
\end{proof}

Let $\Psi_0$ and $\Psi_1$ be the functions that map a vertex $x|\x$ to its adjacent vertex from an edge or an arc, respectively. That is,
\begin{align*}
\Psi_0(x_0|x_1x_2\ldots x_n)&=\overline{x_0}|x_1x_2\ldots x_n,\\
\Psi_1(x_0|x_1x_2\ldots x_n)&=x_0|x_1x_2\ldots (x_0+x_1).
\end{align*}
Let $\Phi=(\phi_1,\phi_2,\ldots,\phi_n)$ be the function that maps every $x_i$ to either $x_i$ or $\overline{x_i}$, for $i=1,2,\ldots,n$.
\begin{lemma}
For any fixed functions $\Psi_j$ and $\Phi$, and  first digit $x_0=0,1$,  we have
$$
\Psi_j(x_0|\Phi(\x))=\Phi(\Psi_j(x_0|\x)),
$$
where $\Phi$ only acts on the digits $x_1,x_2,\ldots,x_n$.
\end{lemma}
\begin{proof}
\begin{align*}
\Psi_0(x_0|\Phi(\x)) &=\Psi_0(x_0|\phi_1(x_1)\phi_2(x_2)\ldots\phi_n(x_n))=\overline{x_0}|\phi_1(x_1)\phi_2(x_2)\ldots\phi_n(x_n)\\
 &=\Phi(\Psi_0(x_0|\x)).\\
 \Psi_1(x_0|\Phi(\x)) &=\Psi_1(x_0|\phi_1(x_1)\phi_2(x_2)\ldots\phi_n(x_n))=x_0|\phi_2(x_2)\ldots\phi_n(x_n)(x_0\phi_1(x_1))\\
 &=\Phi(\Psi_1(x_0|\x)).
\end{align*}
\end{proof}

Another property of the mixed graph $G^+(n)$ for $n>1$ is that from every pair of (not necessarily distinct) vertices $u$ and $v$, there is at least a walk of length $2n$ from $u$ to $v$.
For instance, for $n=2$, fixing as before $x_1=0$ and setting $y=x_0+x_2$, we have the following walks
of length $4$ from  $x_0|0 x_2$ to every vertex of $G^+(2)$.

\begin{align}
x_0|0 x_2 & \,\sim\, \overline{x_0}|0x_2\,\sim\, x_0|0x_2\,\rightarrow\,x_0|x_2x_0\,\rightarrow\, x_0|x_0 (x_0+x_2)
 &= x_0|x_0 y\nonumber\\
          & \,\rightarrow\,  x_0|x_2 x_0\,\sim\, \overline{x_0}|x_2 x_0\,\rightarrow\,\overline{x_0}|x_0(\overline{x_0}+x_2)\sim\, x_0|x_0(\overline{x_0}+x_2)
& = x_0|x_0\overline{y}\nonumber\\
          & \,\sim\, \overline{x_0}|0 x_2\,\rightarrow\, \overline{x_0}|x_2 \overline{x_0}\,\sim\,x_0|x_2 \overline{x_0}\,\rightarrow\, x_0|\overline{x_0} (x_0+x_2)
 & = x_0|\overline{x_0}y         \nonumber\\
          & \,\sim\, \overline{x_0}|0x_2\,\rightarrow\, \overline{x_0}|x_2\overline{x_0}\,\rightarrow\,\overline{x_0}|\overline{x_0}(\overline{x_0}+x_2)\,\sim\, x_0|\overline{x_0}(\overline{x_0}+x_2)
 & = x_0|\overline{x_0}\overline{y} \nonumber\\
           & \,\rightarrow\, x_0|x_2 x_0 \,\rightarrow\,x_0|x_0 (x_0+x_2)\,\rightarrow\, x_0|(x_0+x_2) 0\,\sim\,\overline{x_0}|(x_0+x_2) 0
& = \overline{x_0}|y0 \nonumber\\
& \,\rightarrow\, x_0|x_2 x_0 \,\rightarrow\,x_0|x_0 (x_0+x_2)\,\sim\, \overline{x_0}|(x_0+x_2)\,\rightarrow\,\overline{x_0}|(x_0+x_2) 1
& = \overline{x_0}|y1\nonumber\\
& \,\sim\, \overline{x_0}|0 x_2 \,\rightarrow\,\overline{x_0}|x_2 \overline{x_0}\,\rightarrow\, \overline{x_0}|\overline{x_0}(\overline{x_0}+x_2)\,\rightarrow\,\overline{x_0}|(\overline{x_0}+x_2) 0
& = \overline{x_0}|\overline{y}0 \nonumber\\
& \,\rightarrow\, x_0|x_2 x_0\,\sim\, \overline{x_0}|x_2x_0 \,\rightarrow\, \overline{x_0}|x_0(\overline{x_0}+x_2) 0\,\rightarrow\,\overline{x_0}|(\overline{x_0}+x_2) 1
& = \overline{x_0}|\overline{y}1.\nonumber
\end{align}

Working with the adjacency matrix $\A$ of $G^+(2)$ (indexed according to Figure \ref{fig:G+(2)}), the above property is apparent when we look at the power $\A^4$.

\begin{align*}
\A &=
\left(
\begin{array}{cccccccc}
1 & 1 & 0 & 0 & 0 & 0 & 0 & 0\\
1 & 0 & 1 & 0 & 0 & 0 & 0 & 0\\
0 & 0 & 0 & 1 & 0 & 0 & 1 & 0\\
0 & 0 & 1 & 0 & 1 & 0 & 0 & 0\\
0 & 0 & 0 & 1 & 0 & 1 & 0 & 0\\
0 & 1 & 0 & 0 & 1 & 0 & 0 & 0\\
0 & 0 & 0 & 0 & 0 & 1 & 0 & 1\\
0 & 0 & 0 & 0 & 0 & 0 & 1 & 1
\end{array}
\right),
& \A^4 =
\left(
\begin{array}{cccccccc}
5 & 3 & 3 & 1 & 1 & 1 & 1 & 1\\
3 & 3 & 1 & 3 & 1 & 1 & 3 & 1\\
1 & 1 & 3 & 1 & 3 & 3 & 1 & 3\\
1 & 1 & 1 & 5 & 1 & 3 & 3 & 1\\
1 & 3 & 3 & 1 & 5 & 1 & 1 & 1\\
3 & 1 & 3 & 3 & 1 & 3 & 1 & 1\\
1 & 3 & 1 & 1 & 3 & 1 & 3 & 3\\
1 & 1 & 1 & 1 & 1 & 3 & 3 & 5\\
\end{array}
\right).
\end{align*}

\subsection{The $n$-line mixed graphs}

Let $G=(V,A)$ be a $2$-regular digraph with a given $1$-factorization, that is, containing two arc-disjoint spanning $1$-regular digraphs $H_1$ and $H_2$. Assuming that the arcs of $H_1$ have color blue and the arcs of $H_2$ have color red, we can also think about a (proper) arc-coloring $\gamma$ of $G$. Then, if $xy$ represents an arc of $G$, we denote its color as $\gamma(xy)$.\\   
Given an integer $n\ge 3$, the vertices of the $n$-line mixed graph $H(n)=H_n(G)$ are the set of $n$-walks in $G$,
$x_1x_2\ldots x_{n-1}x_n$, with $x_i\in V$ and $x_ix_{i+1}\in A$, for $i=1,\ldots,n-1$. The adjacencies of $H(n)$ are as follows:
\begin{align}
x_1x_2\ldots x_{n-1}x_n & \sim\  y_1x_2\ldots x_{n-1}x_n \mbox{\  (edges),}
\label{edgesH(n)}
\end{align}
where $\gamma(y_1x_2)\neq \gamma(x_1x_2)$; and
\begin{align}
x_1x_2\ldots x_{n-1}x_n & \rightarrow\ x_2\ldots x_{n-1}x_ny_{n+1} \mbox{\  (arcs),}
\label{arcsH(n)}
\end{align}
where 
$\gamma(x_ny_{n+1})\!=$red if 
$\gamma(x_1x_2)=\gamma(x_{n-1}x_n)$,
and   $\gamma(x_ny_{n+1})\!=$blue if  $\gamma(x_1x_2)\neq\gamma(x_{n-1}x_n)$.
The reason for the name of $H_n(G)$ is because when we contract all its edges, so identifying the vertices in \eqref{edgesH(n)}, the resulting digraph is the $(n-1)$-iterated line digraph $L^{n-1}(G)$ of $G$, see Fiol, Yebra, and Alegre \cite{FiYeAl84}. Indeed, under such an operation, each pair of vertices in \eqref{edgesH(n)} becomes a vertex that can be represented by the sequence $x_2x_3\ldots x_n$, which, according to \eqref{arcsH(n)}, is adjacent to the two vertices $x_3\ldots x_ny_{n+1}$ with $y_{n+1}\in \Gamma^+ (x_n)$ in $G$.
\\
In the following result, we describe other basic properties of $H_n(G)$.
\begin{proposition}
Let $G=(V,A)$ be a digraph with $r$ vertices and diameter $s$, having a $1$-factorization. For a given $n\ge 3$,  the following holds.
\begin{itemize}
\item[$(i)$]
The mixed graph $H_n=H_n(G)$ has $N=r\cdot 2^{n-1}$ vertices,  and it is totally $(1,1)$-regular with no digons.
\item[$(ii)$]
The diameter of $H_n$ satisfies $k\le 2(s+n)-3$. 
\end{itemize}
\end{proposition}
\begin{proof}
$(i)$ Every vertex $x_1\ldots x_n$ of $H_n$ corresponds to a walk of $G$ with first vertex $x_1$, which gives $r$ possibilities and, since $G$ is $2$-regular, for every other $x_i$, $i=2,\ldots,n$, we have $2$ possible options. This provides the value of $N$.\\
To show total $(1,1)$ regularity, it is enough to prove that $H_n$ is $1$-in-regular. Indeed, any vertex adjacent to $x_1x_2\ldots x_n$, with $\gamma(x_{n-1}x_n)$=blue (respectively, $\gamma(x_{n-1}x_n)=$ red) must be of the form 
$yx_1\ldots x_{n-2}x_{n-1}$ with $\gamma(yx_1)\neq \gamma(x_{n-2}x_{n-1})$ (respectively,  with  $\gamma(yx_1)= \gamma(x_{n-2}x_{n-1}))$. But, in both cases, there is only one possible choice for vertex $y$. \\
With respect to the absence of digons, notice that a vertex $\vecu=x_1x_2\ldots x_{n-1}x_n$ belongs to a digon if, after two steps, we come back to $\vecu$, which means that $x_1x_2\ldots x_{n-1}x_n=x_3x_4\ldots x_ny_{n+1}y_{n+2}$ and, hence, $x_i=x_3=\cdots$ and $x_2=x_4=\cdots$. In other words, vertex $\u$ must be of the form $xyxy\cdots xy$ ($n$ even) or $xyxy\cdots x$ ($n$ odd), and $G$ itself must have a digon between vertices $x$ and $y$. Assuming that $n$ is even and $\gamma(xy)$=blue (the other cases are similar), the digon should be
$$
\vecu=xyxy\ldots xy\quad \rightarrow\quad \vecv=yxyx\ldots yx\quad \rightarrow\quad \vecu. 
$$
But the last adjacency is not possible since both the first and last arcs of $\vecv$ would have color $\gamma(yx)=$red and, hence, so should be the color of $xy$, a contradiction.
\\
$(ii)$ Given both vertices $x_1x_2\ldots x_{n-1}x_n$ and $y_1y_2\ldots y_{n-1}y_n$, let us consider a shortest path in $G$ of length at most $s$ from $x_n$ to $y_2$. Then, using both types of adjacencies, we can go from $x_1x_2\ldots x_{n-1}x_n$ to a vertex of the form $z_1\ldots y_2$. From this vertex, we now reach the vertex $yy_2\ldots y_n$ in at most $2(n-2)$ steps. Finally, if necessary,
we can change $y$ by $y_1$. In total, we use $k\le 2s+2(n-2)+1=2(s+n)-3$ steps, as claimed.
\end{proof}

For example, if $G$ is the complete symmetric digraph $K_3$ (edges seen as digons) with vertices in $\mathbb{Z}_3$, blue arcs $i\rightarrow i+1$ and red arcs
$i\rightarrow i-1$ for $i=0,1,2$, the adjacencies of $H_n(K_3)$, with $3\cdot 2^{n-1}$ vertices, are
\begin{align*}
x_1x_2\ldots x_{n-1}x_n & \sim\  y_1x_2\ldots x_{n-1}x_n,\quad y_1\neq x_1,x_2,\\
x_1x_2\ldots x_{n-1}x_n & \rightarrow\ x_2x_3\ldots x_n y_{n+1},\quad y_{n+1}=x_n-(x_2-x_1)(x_n-x_{n-1}).
\end{align*}
Thus, the $(1,1)$-regular mixed graphs $H_3(K_3)$ and $H_4(K_3)$, with diameter $k=5$ and $k=6$, respectively, are shown in Figure \ref{fig:H(3)}.
\begin{figure}[t]
    \begin{center}
        \includegraphics[width=14cm]{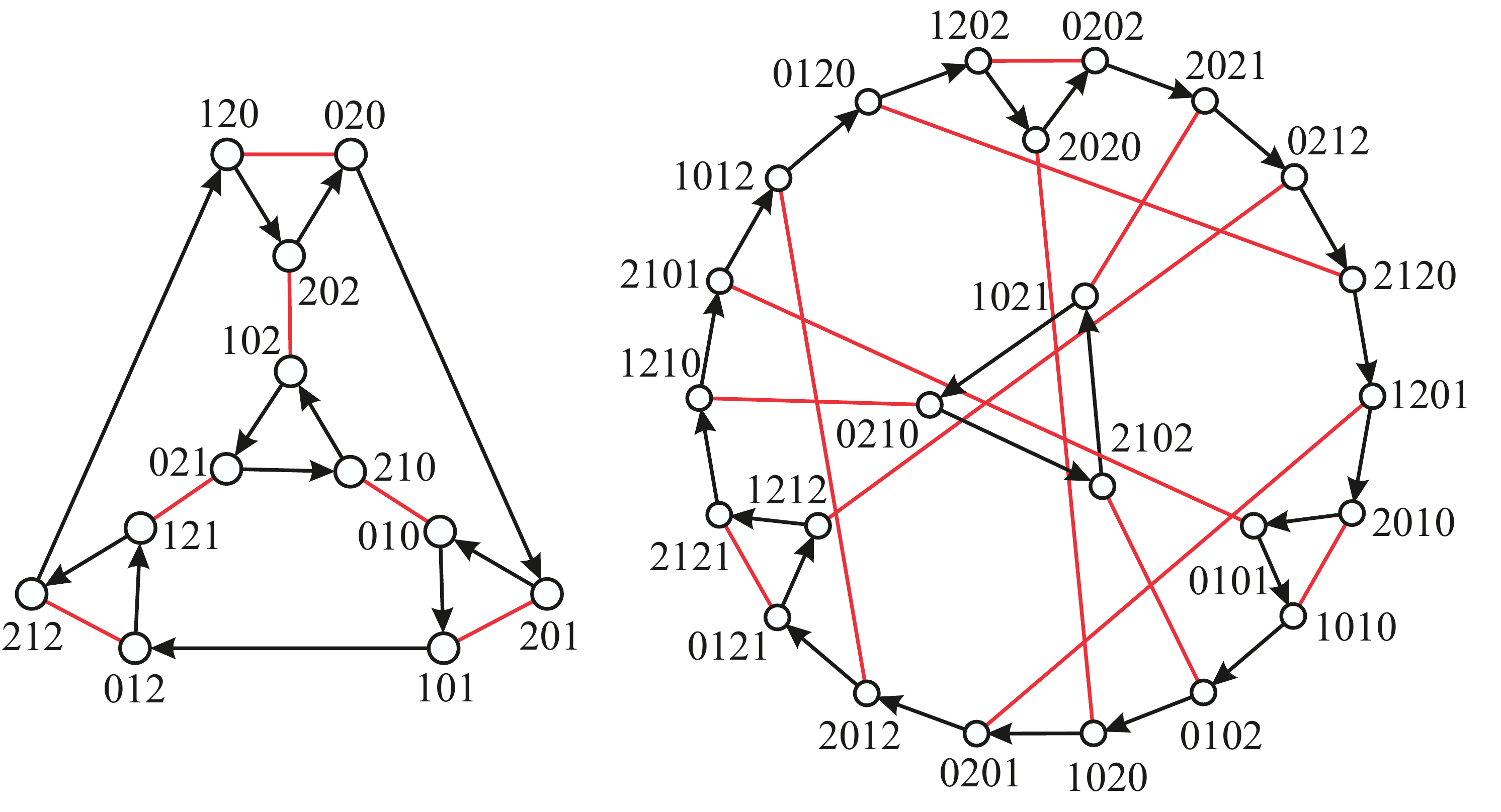}
    \end{center}
    \vskip-.5cm
	\caption{The mixed graphs $H_3(K_3)$ and $H_4(K_3)$.}
	\label{fig:H(3)}
\end{figure}
In this case, when we contract all the edges of $H_n(K_3)$, we obtain the $(n-1)$-iterated line digraph of $K_3$, which, as commented in Introduction, is isomorphic to the Kautz digraph $K(2,n-1)$.

\section{A first computational approach: The $(1,1,k)$-mixed graphs with diameter at most 6}

The Moore bound $M(1,1,k)$ coincides with the number of binary words of length $\ell\leq k$ without consecutive zeroes. In this sense, the corresponding Moore tree can be rooted to a vertex labeled with the empty word. Every vertex labeled with a word $\omega$ (of length $\ell$) with the last symbol different from 0 is joined by an edge to a vertex labeled $\omega 0$ (of length $\ell+1$), for all $0 \leq \ell \leq k-1$. Moreover, the arcs are defined by $\omega \rightarrow \omega 1$ (see an example in Figure \ref{fig:mt}).

This new description of the Moore tree is very useful for performing an exhaustive computational search of the largest mixed graphs for some small values of the diameter $k$. Let $a(\ell)$ be the number of vertices at distance $\ell$ from the root in the Moore tree. Using the above-mentioned labeling, it is easy to see that $a(\ell)$ satisfies the recurrence equation
\begin{equation}
a(\ell)=a(\ell-1)+a(\ell-2),
\end{equation}
with initial conditions $a(0)=1$ and $a(1)=2$. Indeed, $a(\ell)$ is the number of words of length $\ell$ (whose symbols are in the alphabet $\Sigma=\{0,1\}$) without consecutive zeroes. The words of length $\ell$ non-ending with 0 are constructed by a word of length $\ell-1$ by adding $0$. This gives $a(\ell-1)$. Moreover, the words  of length $\ell$ 
ending with $0$ 
are constructed by adding $1$.  This gives $a(\ell-2)=b(\ell)$, 
where $b(\ell)$ is the number of vertices at distance $\ell$ from the root joined by an edge to a vertex at distance $\ell-1$. 
So $b(\ell)$ satisfies the same recurrence relation as $a(\ell)$ but with initial conditions $b(0)=0$ and $b(1)=1$. Finally, let $c(\ell)=a(\ell)-b(\ell)=a(\ell-1)$, that is, the number of vertices at distance $\ell$ from the root pointed by an arc from a vertex at distance $\ell-1$. Again, $c(\ell)$ satisfies the same type of recurrence relation but with initial conditions $c(0)=1$ and $c(1)=1$. Thus, $a(\ell)$, $b(\ell)$, and $c(\ell)$ are all Fibonnaci-like numbers. For instance, $a(\ell)$ equals  the following closed formula 
$$
a(\ell)=\frac{5 + 3\sqrt{5}}{10}\left(\frac{1+\sqrt{5}}{2}\right)^{\ell} + \frac{5 - 3\sqrt{5}}{10}\left(\frac{1-\sqrt{5}}{2}\right)^{\ell}.
$$
Note that the sequence obtained from $a(\ell)$ corresponds to the Fibonacci numbers starting with $a(0)=1$ and $a(1)=2$ (see the sequence A000045 in \cite{OIES}). 
Similar formulas can be obtained for $b(\ell)$ and $c(\ell)$.

Now, we can perform an algorithmic exhaustive search to find all the largest $(1,1,k)$-mixed graphs with order close to the Moore bound. For instance, in the case of almost mixed Moore graphs (with diameter $k$ and order $M(1,1,k)-1$), the number of different cases of mixed graphs to analyze is bounded by ${\cal N}(k)$, where ${\cal N}(k)$ is computed next.
\begin{enumerate}
 \item 
 We remove a vertex in the Moore tree at distance $k$ from the empty word. Notice there are $a(k)$ different choices for this vertex.
 \item 
 Now, we count the number ${\cal N}_1$ of possibilities to complete the undirected part of the mixed graph. We recall that the number of perfect matchings in a complete graph of even order $n$ is $(n-1)!!$ This number ${\cal N}_1$ depends on what  vertex has been removed in the previous step. If the removed vertex has a label ending with $0$, that is, it is a vertex hanging from an edge, then there are $c(k)+1$ vertices in the graph without an incident edge. So, ${\cal N}_1=c(k)!!$  Otherwise,  there are $c(k)-1$ vertices in the graph without an incident edge, so ${\cal N}_1=(c(k)-2)!!$
 \item 
 The number of possibilities to complete the directed part of the graph is upper bounded by the number of mappings from the set of words of length $k$ without fixing points. This is precisely the number of derangements $D_{a(k)}$. Notice that mappings, including assignations from a word of length $k$ to its predecessor, are not valid.
\end{enumerate}

\begin{figure}[t]
	\begin{center}
		\includegraphics[scale=0.4]{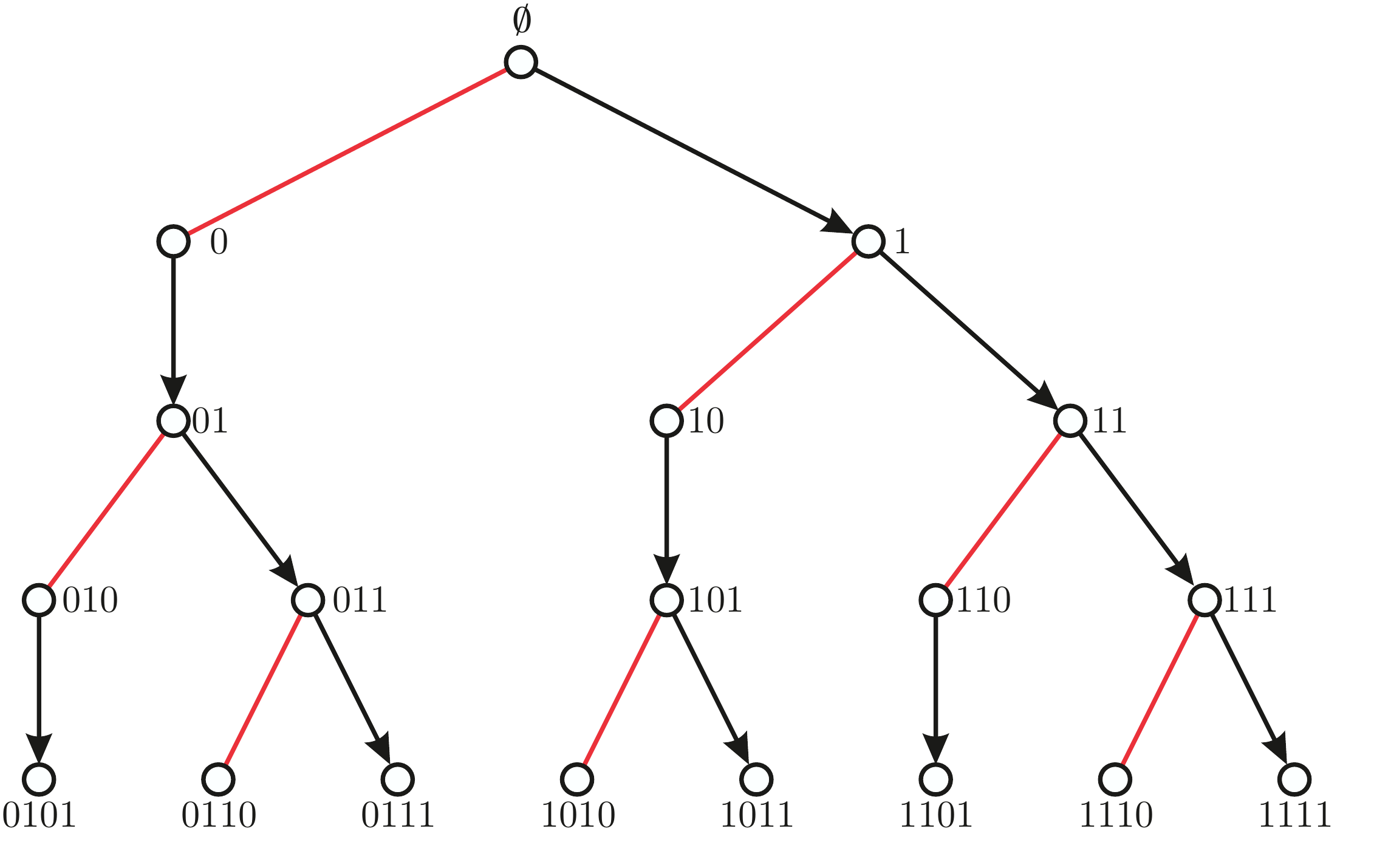}
	\end{center}
\vskip-.5cm
	\caption{The Moore tree with parameters $r=z=1$ and depth $4$ labeled with binary words without consecutive zeros.}
	\label{fig:mt}
\end{figure}

Putting all together, ${\cal N}(k)<D_{a(k)}\left(b(k)c(k)!!+c(k)(c(k)-2)!!\right)$. Of course, ${\cal N}(k)$ grows very fast with $k$ but the number of cases to analyze for $k\leq 4$ is reasonable (see Table \ref{tab:count}).

\begin{table}[t]
\begin{center}
\begin{tabular}{|l||r|r|r|r|}
\hline
$k$ & $3$ & $4$ & $5$ & $6$\\ \hline
$N(k)$ & $396$ & $889980$ & $0$ & $
2\cdot 10^{25}$ \\\hline
\end{tabular}
\end{center}
\caption{Values of ${\cal N}(k)$. Notice that ${\cal N}(5)=0$ because $c(5)$ is an even number.}\label{tab:count}
\end{table}

As a consequence, computing the diameter of the $889980$ putative almost Moore $(1,1)$-mixed graphs with diameter $k=4$, we have the following result. (In fact, this calculation is easily done by using the result by Tuite and Erskhine \cite{te22} that such  graphs are not totally regular.)

\begin{proposition}
There is no almost $(1,1,4)$-mixed Moore graph.
\end{proposition}
A similar method can be implemented to perform an exhaustive search for orders $M(1,1,k)-\delta$ for small $\delta$. In these cases, the removal of $\delta$ different vertices of the Moore tree (step 1) has many more choices but the number of operations in steps 2 and 3 sometimes is reduced. This is precisely what we do for $n=M(1,1,4)-3=16$, where there are two cases to take into account:
\begin{itemize}
 \item The removal of three distinct words $\omega_1,\omega_2,\omega_3$ of length $4$ (corresponding to three distinct vertices at distance $4$ from the root of the Moore tree).
 \item Given any word $\omega$ of length $3$, the deletion of either the set of words $\{\omega,\omega1,\omega'\}$ (when $\omega$ ends in $0$) or $\{\omega,\omega0,\omega1\}$ (when $\omega$ ends in $1$), where $\omega'\neq \omega1$ is any word of length $4$.
 \end{itemize}

It remains to add the corresponding edges and arcs in the pruned Moore tree. The computational exhaustive search shows there is no $(1,1)$-mixed Moore graph of diameter $4$ and order $16$. Now the maximum order becomes $n=14$ for a mixed graph with parameters $r=z=1$ and $k=4$. There are many more possibilities to prune the Moore tree, so we decide to implement a direct method to perform an exhaustive search in this case: taking the perfect matching with a set of vertices $V=\{0,1,\dots,13\}$ and where $i \sim i+1$ for all even $i$, we add the three arcs $(0,2),(1,5)$ and $(5,7)$. Looking at vertex $0$ as the root of the Moore tree, the existence of these three arcs in the mixed graph is given because $\delta=5$ in this case. Now we proceed with the exhaustive search by adding the remaining arcs in the graph. There are $11!$ possibilities but excluding avoided permutations (those permutations with elements of order at most $2$ or including edges of the perfect matching) significantly reduces the number of cases to analyze. After computing the diameter of all these mixed graphs and keeping those non-isomorphic mixed graphs with diameter $k=4$, we have the following result. 

\begin{proposition}
\label{k=4}
The maximum order for a $(1,1)$-mixed regular graph of diameter $k=4$ is $14$. There are $27$ of such mixed graphs (see Table \ref{tab:optn14}), and only one of them is a Cayley graph. Namely, that of the dihedral group $D_7$ with generators $r$ and $s$, and presentation $\langle r,s\, |\, r^7\!=\!s^2\!=\!(rs)^2\!=\!1 \rangle$, also obtained as the line digraph of $C_7$, see the mixed graph at the top left in Figure \ref{fig3}.
\end{proposition}

\begin{table}[t]
\begin{center}
\begin{scriptsize}
\begin{tabular}{|c|c|c|}
\hline
\text{\texttt{MW?H??GC@{\char`\_}?EAO??E?{\char`\_}O?B@{\char`\_}??L{\char`\_}??W?@{\char`\_}}}
 & \text{\texttt{MW?H??K??{\char`\_}QC??o?I??oCC??oGH@?ACH??}} & \text{\texttt{MW?H??K??{\char`\_}QC??o?@c?{\char`\_}CC??oE?I??EH??}}  \\
 \text{\texttt{MW?H??K??{\char`\_}QC??o?B?A{\char`\_}CC??sC?AAACH??}} & \text{\texttt{MW?H??K??{\char`\_}QC??o?BC?{\char`\_}CC??oCIA??K@{\char`\_}?}} & \text{\texttt{MW?H??G@@{\char`\_}?E?OA?E?{\char`\_}O?B?oG?OCG?KE??}} \\
 \text{\texttt{MW?H??G@@{\char`\_}?E?OA?E?{\char`\_}O?B@{\char`\_}G?O?W?WE??}} & \text{\texttt{MW?H??G@@{\char`\_}?E?OA?E?{\char`\_}O?B{\char`\_}{\char`\_}?OO?W?WE??}} & \text{\texttt{MW?H??G@@{\char`\_}?EAO??E?{\char`\_}O?B?{\char`\_}{\char`\_}OO?W?WE??}} \\
 \text{\texttt{MW?H??G@@{\char`\_}?EAO??E?{\char`\_}O?B?{\char`\_}c?O?W?WAO?}} & \text{\texttt{MW?H??G@@{\char`\_}?EAO??E?{\char`\_}O?B@{\char`\_}??W?W?WE??}} & \text{\texttt{MW?H??G@@{\char`\_}?EAO??E?{\char`\_}O?B@{\char`\_}??X?G?WA@?}} \\
 \text{\texttt{MW?H??G@@{\char`\_}?EAO??E?{\char`\_}O?BO{\char`\_}??W?W?WAO?}} & \text{\texttt{MW?H??G@@{\char`\_}?EAO??E?{\char`\_}O?BO{\char`\_}??WCG?WA@?}} & \text{\texttt{MW?H??GC@{\char`\_}?E?OA?E?{\char`\_}O?B?o?OD{\char`\_}?{\char`\_}G?@{\char`\_}}} \\
\text{\texttt{MW?H??GC@{\char`\_}?E?OA?E?{\char`\_}O?B@{\char`\_}?CD{\char`\_}?{\char`\_}G?@{\char`\_}}} & \text{\texttt{MW?H??GC@{\char`\_}?E?OA?E?{\char`\_}O?B@{\char`\_}G?D{\char`\_}??W?@{\char`\_}}} & \text{\texttt{MW?H??GC@{\char`\_}?E?OA?E?{\char`\_}O?B{\char`\_}{\char`\_}?OD{\char`\_}??W?@{\char`\_}}} \\
\text{\texttt{MW?H??GC@{\char`\_}?E?P??E?{\char`\_}O?B{\char`\_}{\char`\_}??L{\char`\_}??K?O{\char`\_}}} & \text{\texttt{MW?H??GC@{\char`\_}?EAO??E?{\char`\_}O?B?o??Kc??KC?{\char`\_}}} &
\text{\texttt{MW?H??K??OQG?@{\char`\_}?E?OO?B{\char`\_}{\char`\_}C?O?IAG@@?}} \\
\text{\texttt{MW?H??K??{\char`\_}Q@?@{\char`\_}?E?OO?BO{\char`\_}??WGG?WH??}} & \text{\texttt{MW?H??K??{\char`\_}U??@{\char`\_}?E?OO?B?{\char`\_}{\char`\_}{\char`\_}O?W?WH??}} & \text{\texttt{MW?H??K??{\char`\_}U??@{\char`\_}?E?OO?B?{\char`\_}g?O?W?W@{\char`\_}?}} \\
\text{\texttt{MW?H??K??`A@?@{\char`\_}?E?OO?BO{\char`\_}??MO?AG?C{\char`\_}}} & \text{\texttt{MW?H??GO@{\char`\_}?E?OO?B?{\char`\_}O?BW???MA?@G?C{\char`\_}}} & \text{\texttt{MW?H??GC@{\char`\_}?E?OO?A{\char`\_}{\char`\_}O?B{\char`\_}?OK?c?OG?@{\char`\_}}} \\
 \hline
\end{tabular}
\end{scriptsize}
\end{center}
\caption{All $27$ mixed graphs with largest order for $r=z=1$ and diameter $k=4$ given as an ASCII string encoded in \emph{digraph6} representation (see McKay and  Piperno \cite{nauty}). The first one in the list is precisely the Cayley graph depicted in the top left of Figure \ref{fig3}.}\label{tab:optn14}
\end{table}

\begin{figure}[t]
	\begin{center}
		\includegraphics[width=14cm]{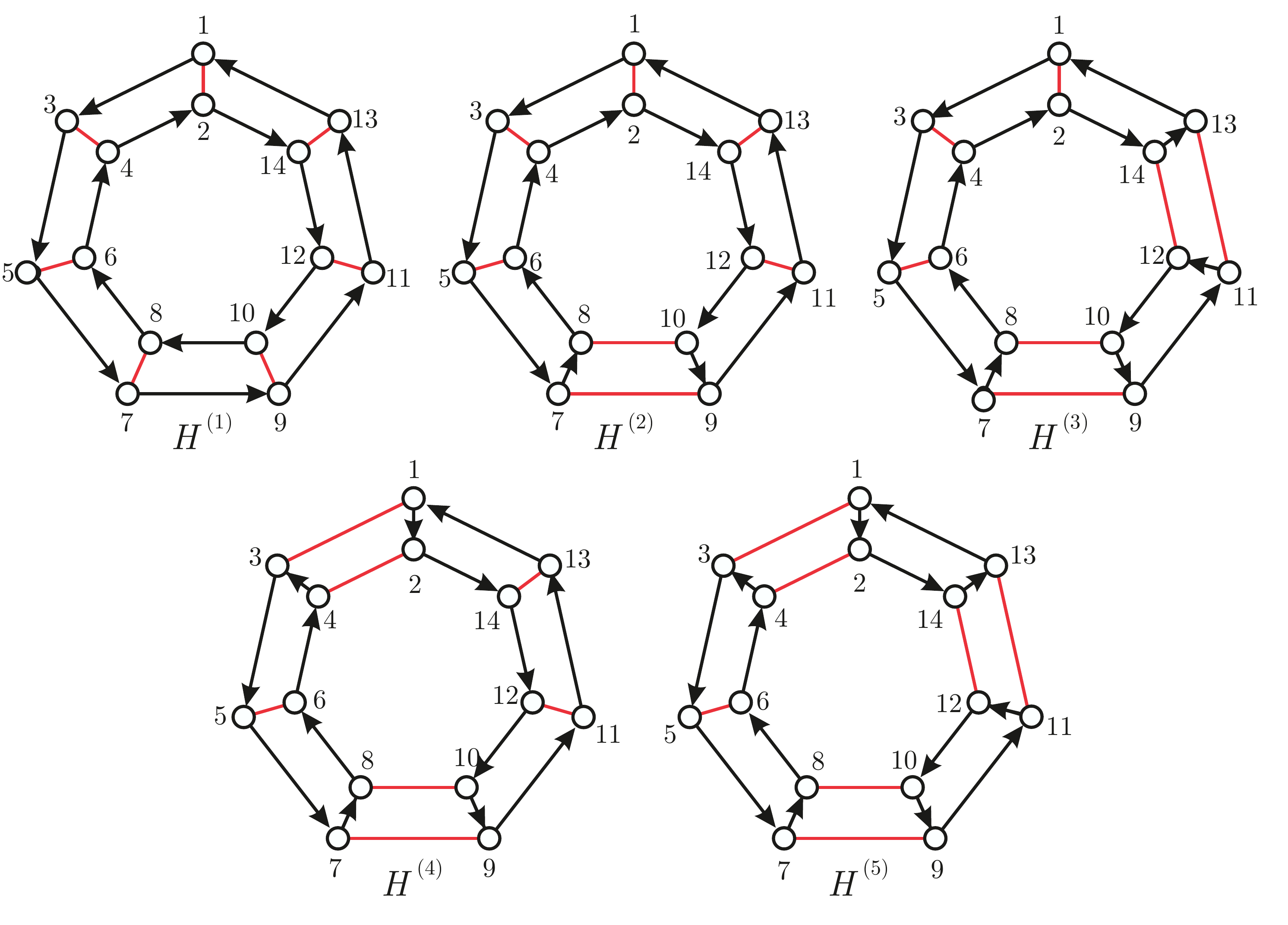}
	\end{center}
	\vskip-.5cm
	\caption{Some cospectral largest mixed graphs with diameter $k=4$ and degrees $r{=}z{=}1$.}
	\label{fig3}
\end{figure}
The spectra of all $27$ mixed graphs with the largest order can be described with the help of the (complex) roots $\alpha_{ij}$ of the irreducible polynomials $p_i(x) \in \mathbb{Q}(x)$ given below:
\[
\begin{array}{lll}
p_1(x)=x^4 + x^3 - 2x^2 - x + 2, \\
p_2(x)=x^3 + x^2 - 2x - 1, \\
p_3(x)=x^3-x+1, \\
p_4(x)=x^3 + 2x^2 - x - 3, \\
p_5(x)= x^4 + x^3 - x^2 - x + 1, \\
p_6(x)=x^6 + x^5 - 3x^4 - x^3 + 5x^2 - 4, \\
p_7(x)=x^9 + 3x^8 - 6x^6 + 2x^5 + 11x^4 - 3x^3 - 9x^2 + 3x + 3, \\
p_8(x)=x^6 + x^5 - 3x^4 - 2x^3 + 5x^2 + 2x - 3, \\
p_9(x)=x^6 + x^5 - x^4 + 3x^2 - 1.
\end{array}
\]

\begin{table}[htb]
\begin{center}
\begin{tabular}{|c|l|}
\hline
Number of graphs & Spectra \\
\hline
9 & $\{2^1,0^6,\alpha_{1j}^1,\alpha_{2s}^1\}$  for $j=1,2,3,4$ and $s=1,2,3$\\
6 & $\{2^1,-1^1,1^1,0^5,\alpha_{3j}^1,\alpha_{4j}^1\}$  for $j=1,2,3$\\
5 & $\{2^1,0^7,\alpha_{2j}^2\}$  for $j=1,2,3$\\
4 & $\{2^1,0^3,\alpha_{5j}^1,\alpha_{6k}^1\}$  for $j=1,2,3,4$ and $k=1,\ldots,6$\\
2 & $\{2^1,1^1,0^3, \alpha_{7j}^1\}$  for $j=1,\ldots,9$\\
1 & $\{2^1,0^1,\alpha_{8j}^1,\alpha_{9j}^1\}$  for $j=1,\ldots,6$\\
\hline
\end{tabular}
\end{center}\caption{Classification of the largest mixed graphs for $r=z=1$ and diameter $k=4$ according to their spectra. The third row gives the spectrum of the 5 cospectral graphs in Figure \ref{fig3}.}
\end{table}

\section{A second computational approach: Cayley or lift $(1,1,k)$-mixed graphs with small diameter}

To obtain the results of this section, we followed a different strategy. We mainly concentrate our search on looking at large $(1,1,k)$-mixed graphs that are either Cayley or lift graphs. Let us first recall these two classes of graphs.

Given a finite group $\Omega$ with generating set $S\subseteq \Omega$, the {\em Cayley graph}
$\Cay (\Omega,S)$ has vertices representing the elements of $\Omega$, and arcs from $\omega$ to $\omega s$ for every $\omega\in \Omega$ and $s\in S$.
Notice that if $s,s^{-1}\in S$, then we have an edge (as two opposite arcs) between $\omega$ and $\omega s$.
Thus, if $S=S_1\cup S_2$ where $S_1=S_1^{-1}$ and $S_2\cap S_2^{-1}=\emptyset$, the Cayley graph $\Cay(\Omega,S)$ is an $(r,z)$-mixed graph with undirected degree $r=|S_1|$ and directed degree $z=|S_2|$.

Given a digraph $G$, or {\em base graph}, and a finite group $\Omega$ with generating set $S$, a {\em voltage assignment $\alpha$} is a mapping $\alpha:E\rightarrow S$, that is, a labeling of the arcs with the elements of $S$.
Then, the {\em lift digraph} $G^\alpha$ has vertex set  $V(G^\alpha)=V\times \Omega$ and
arc set $E(G^\alpha)=E\times S$, where  there is an arc from vertex $(u,g)$ to vertex $(v,g\alpha(uv))$ if and only if $uv\in E$.
In particular, the Cayley digraph $\textrm{Cay}(\Omega,S)$ with $S=\{g_1,\ldots,g_{r}\}$ can be seen as the lifted digraph
$G^\alpha$, where $G=K_1^{r}$ (a singleton with $V=\{u\}$ and $E=\{e_1,\ldots,e_{r}\}$ are $r$ loops) and voltage assignment
$ \alpha(e_i)=g_i$ for $i=1,\ldots,r$.
An example of a lift digraph is shown in Figure \ref{fig:lift}.

\begin{figure}[t]
	\begin{center}
		\includegraphics[width=10cm]{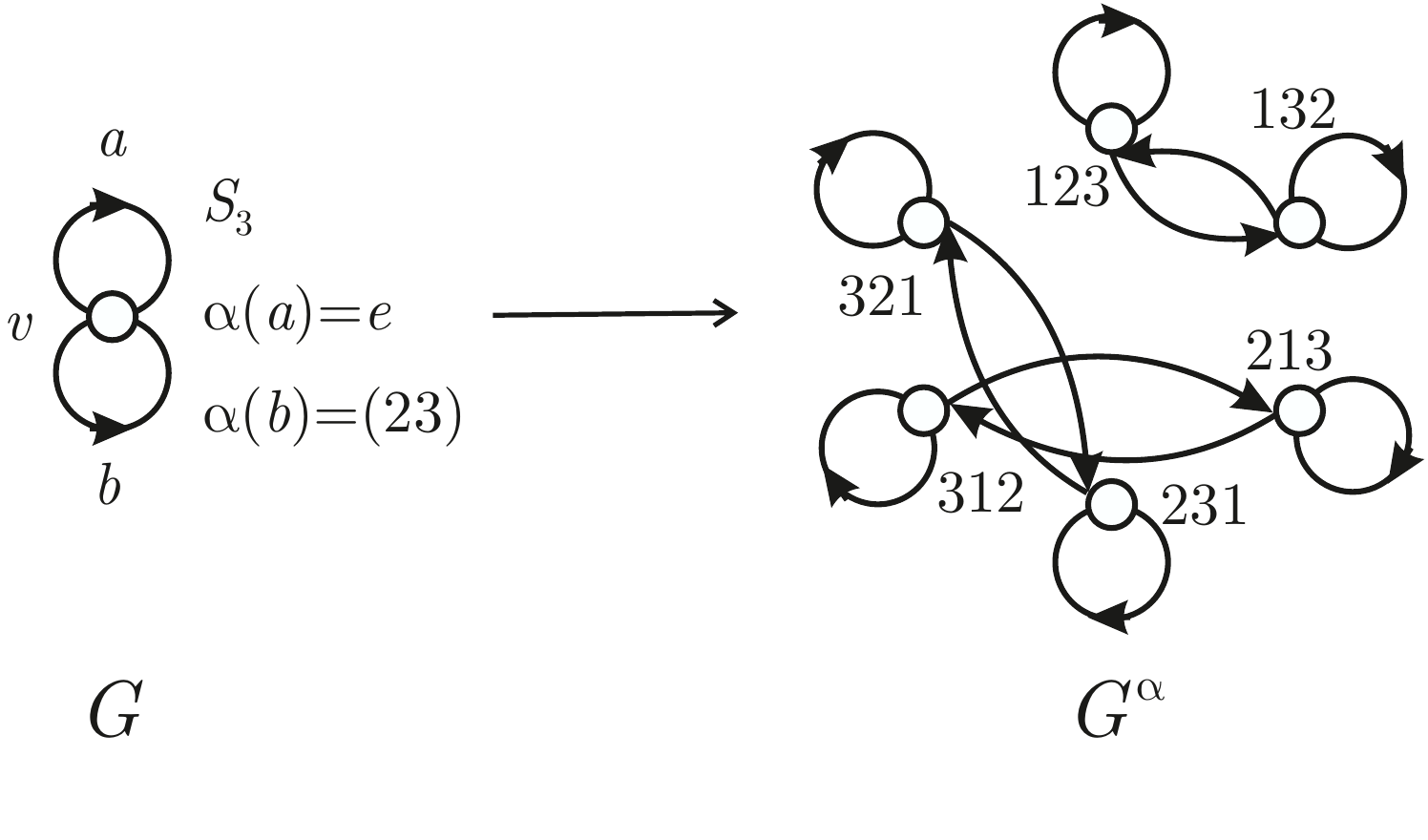}
	\end{center}
 \vskip-1cm
	\caption{A base graph with voltages on the symmetric group $S_3$, and its lift graph.}
	\label{fig:lift}
\end{figure}

The results obtained by computer search are shown in Table \ref{table4}, see next section. In what follows, we comment upon some of the cases.
Notice that for diameter $k=2,3,4$, the known $(1,1,k)$-mixed graphs have the maximum possible order. 
The mixed graph of diameter $k=2$ is the Kautz digraph $K(2,2)$.
The graph with $k=3$ is isomorphic to the line digraph of the cycle $C_5$. Some of the maximal graphs with diameter $k=4$ were already shown in Figure \ref{fig3}.

Two maximal graphs of diameter $k=5$ are shown in Figure \ref{fig4}. 

\begin{figure}[t]
	\begin{center}
		\includegraphics[width=14cm]{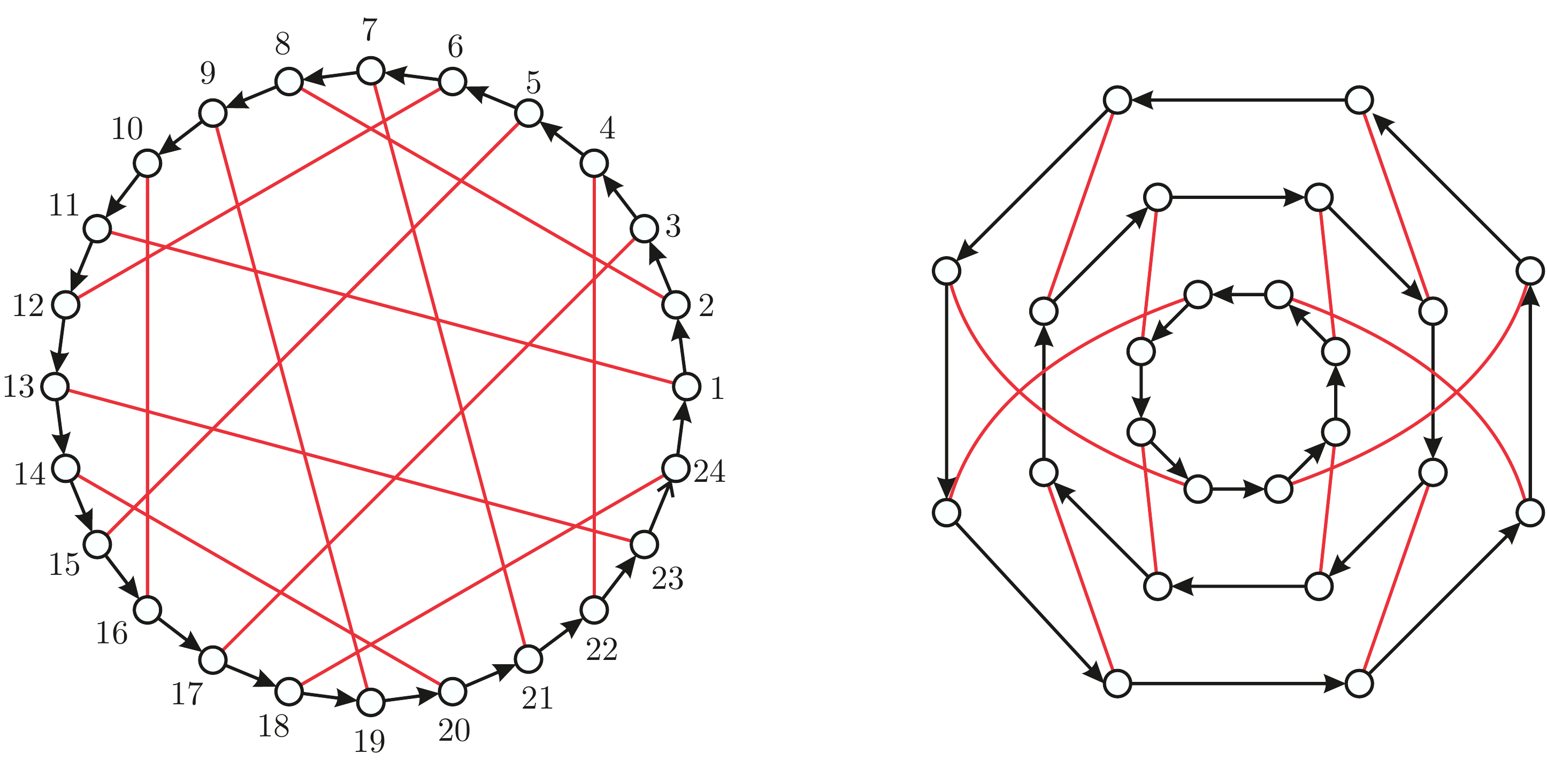}
	\end{center}
	\vskip-.5cm
	\caption{Two $(1, 1, 5)$-mixed graphs with defect 8 and order 24.}
	\label{fig4}
\end{figure}


The graph of order $72$ listed in the table for $k=8$ is a lift graph using
the dihedral group $D_{18}$ of order $18$.  This group consists of the $18$
symmetries of the nonagon. To describe our graph, we consider a
regular nonagon whose vertices are labeled $0$ to $8$ in clockwise
order.  Label the elements of $D_{18}$ as follows.
There are nine counter-clockwise rotations, each through an angle $2 \pi k/9$
and denoted {\tt Rot(k)}, for $0 \leq k < 9$.  Finally, there are the
nine reflections {\tt Ref(k)} about the line through vertex $k$ and
the midpoint of the opposite side.  This notation is
used to specify the voltages on the edges and arcs of the base graph
shown in  Figure \ref{fig:k=8}.
\begin{figure}[t]
\centering
\begin{tikzpicture}[scale=1.35]
\GraphInit[vstyle=Normal]
\tikzset{->-/.style={
        decoration={
            markings,
            mark=at position #1 with {\arrow[scale=2]{>}}
        },
        postaction={decorate}
    },
    ->-/.default=0.9
}
\tikzset{
    LabelStyle/.style = {rectangle, rounded corners, draw,
                        minimum width = 2em, fill = white,
                        text = black, font = \bfseries },
    VertexStyle/.append style = {
        inner sep=4pt,
        ultra thick,
        font=\Large\bfseries,
        fill=white
    }
} 
\Vertex[x=0,y=0]{D}
\Vertex[x=3,y=0]{C}
\Vertex[x=3,y=3]{B}
\Vertex[x=0,y=3]{A}
\tikzset{
    EdgeStyle/.append style = {
        ultra thick
    }
} 
\Edge[label=Rot(3)](A)(D)
\Loop[dist=15mm,dir=EA,style={ultra thick},label=Ref(1)](B)
\Loop[dist=15mm,dir=EA,style={ultra thick},label=Ref(4)](C)
\tikzset{
    EdgeStyle/.append style = {
        ->-, ultra thick
    }
} 
\Edge[label=Rot(8)](A)(B)
\Edge[label=Ref(2)](B)(C)
\Edge[label=Ref(0)](C)(D)
\Loop[dist=15mm,dir=WE,style={->-,ultra thick},label=Rot(4)](D)
\end{tikzpicture}
\vskip-1cm
\caption{The base graph of the $(1,1,8)$-mixed graph of order $72$.}
\label{fig:k=8}
\end{figure}
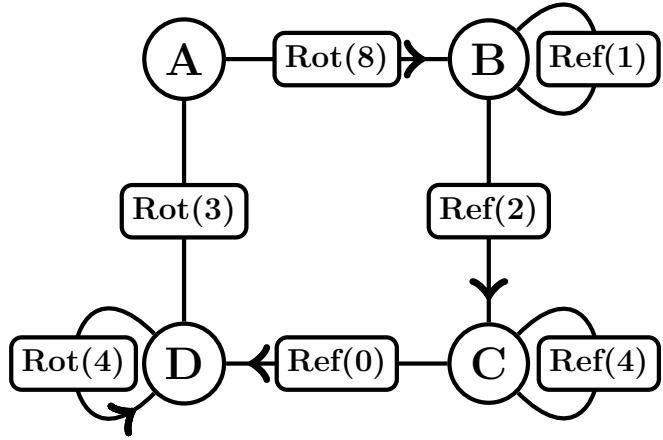

The graph of order $544$ for $k=13$ is a lift of the base graph shown in Figure~\ref{fig:k=13} with voltages in the group $\mathbb{Z}_{17}:\mathbb{Z}_8$.

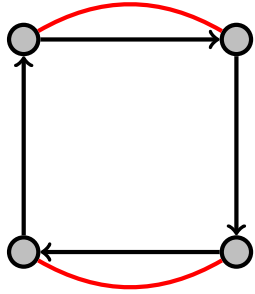
\begin{figure}[t]
    \centering
    \begin{tikzpicture}[x=0.2mm,y=0.2mm,ultra thick,vertex/.style={circle,draw,minimum size=10,fill=lightgray}]
	\node at (-70.7,70.7) [vertex] (v1) {};
	\node at (70.7,70.7) [vertex] (v2) {};
	\node at (70.7,-70.7) [vertex] (v3) {};
	\node at (-70.7,-70.7) [vertex] (v4) {};
	\draw [->] (v1) to (v2);
	\draw [bend left,red] (v1) to (v2);
	\draw [->] (v2) to (v3);
	\draw [->] (v3) to (v4);
	\draw [bend left,red] (v3) to (v4);
	\draw [->] (v4) to (v1);
    \end{tikzpicture}
    \caption{The base graph of the $(1,1,13)$-mixed graph of order $544$.}
    \label{fig:k=13}
\end{figure}




The remaining graphs are partially identified as notes following Table \ref{table4}. Where a graph is identified as a lift using a voltage group of order half the order of the graph, the base graph is an undirected edge together with a directed loop at each vertex.
A complete description of such larger graphs, especially those that use unfamiliar groups, would take a lot of pages. The interested reader can address the third author to request more information.
%
%
%

\section{Table of large $(1,1,k)$-mixed graphs}

A summary of the results for a $(1,1)$-regular mixed graphs with diameter $k$ at most $16$  is shown in Table \ref{table4}, where the lower bounds come from the mentioned constructions. Moreover, the upper bounds follow by Proposition \ref{k=4} $(k=4)$, a computer exploration $(k=5)$, and the numbers $M(1,1,k)-\delta(k)$ with $\delta(k)$ given in \eqref{improved-defect-recur} and adjusted even parity (since $r=1$, the graph contains a perfect matching and, so, it must have even order), see Tuite and Erskine \cite{te22}.


\begin{table}[ht]
\begin{center}
\begin{tabular}{|r||r|r|r|l|} \hline
$k$ & Lower bound & Upper bound & Moore $M(1,1,k)$ & Notes \\ \hline \hline
 2 &  6 &  6 &  6 & \\ \hline
 3 & 10 & 10 & 11 & \\ \hline
 4 & 14 & 14 & 19 & \\ \hline
 5 & 24 & 26 & 32 & \\ \hline
 6 & 34 & 48 & 53 & \\ \hline
 7 & 54 & 78 & 87 & Cayley\footnote{Cayley graph on SmallGroup(54,6): $Z_9 : Z_6$.} \\ \hline
 8 & 72 & 126 & 142 & Lift\footnote{Lift group is the dihedral group of order 18.} \\ \hline
 9 & 112 & 206 & 231 & Lift\footnote{Lift group is $AGL(1,8)$.}  \\ \hline
10 & 144 & 336 & 375 & Cayley\footnote{Cayley graph on SmallGroup(144,182).}  \\ \hline
11 & 240 & 544 & 608 & Lift\footnote{Lift group $A_5\times Z_2$}  \\ \hline
12 & 336 & 882 & 985 & Lift\footnote{Cayley graph on $PSL(2,7):Z_2$.}  \\ \hline
13 & 544 & 1428 & 1595 & Lift\footnote{Lift group $Z_{17} : Z_{8}$} \\  \hline
14 & 800 & 2312 & 2582 & Cayley\footnote{group SmallGroup(800,1191).}  \\ \hline
15 & 1024 & 3744 & 4179 & Lift\footnote{Lift group SmallGroup(512,1727)} \\ \hline
16 & 1600 & 6058 & 6763 & Lift\footnote{Lift group SmallGroup(800,1191)}\\
\hline
\end{tabular}
\caption{Bounds for mixed graphs with $(r,z,k) = (1,1,k)$.}
\label{table4}
\end{center}
\end{table}

\begin{enumerate}
\item 
Cayley graph on SmallGroup(54,6): $\mathbb{Z}_9 : \mathbb{Z}_6$. 
\item 
Lift group is the dihedral group of order 18.
\item 
Lift group is $AGL(1,8)=(\Z_2^3):\Z_7$.
\item 
Cayley graph on SmallGroup(144,182).
\item 
Lift group is $A_5\times \mathbb{Z}_2$.
\item 
Cayley graph on $PSL(2,7):\mathbb{Z}_2$.
\item 
Lift group is $\mathbb{Z}_{17} : \mathbb{Z}_{8}$.
\item 
Cayley graph on SmallGroup(800,1191).
\item 
Lift group is SmallGroup(512,1727).
\item 
Lift group is SmallGroup(800,1191).
\end{enumerate}

\section{Statements \& Declarations}

\subsection{Funding}

The research of C. Dalf\'o, M. A. Fiol, N. L\'opez, and A. Messegu\'e has been supported by
AGAUR from the Catalan Government under project 2021SGR00434 and MICINN from the Spanish Government under project PID2020-115442RB-I00. The research of M. A. Fiol was also supported by a grant from the  Universitat Polit\`ecnica de Catalunya with references AGRUPS-2022 and AGRUPS-2023. J. Tuite was supported by EPSRC grant EP/W522338/1.

\subsection{Competing Interests}

The authors have no relevant financial or non-financial interests to disclose.

\subsection{Author Contributions}

All authors contributed to the study's conception and design. Material preparation, data collection, and analysis were performed by all the authors, after much work was done. All authors contributed to the first draft of the manuscript, which was improved by all of them. All authors read and approved the final manuscript.

\subsection{Data availability}

The datasets generated during and/or analyzed during the current study are available from the corresponding author upon reasonable request.


\end{document}